\documentclass[11pt, reqno ,a4]{amsart}

\usepackage{lmodern}
\usepackage[utf8]{inputenc}
\usepackage{amsmath}
\usepackage{graphicx}
\usepackage{amssymb}
\usepackage{esint}
\usepackage[dvipsnames]{xcolor}

\usepackage{pgfplots}
\pgfplotsset{compat=1.15}
\usepackage{mathrsfs}
\usetikzlibrary{arrows}

\usepackage{xxcolor}
\usepackage{floatrow}
\usepackage{color}
\usepackage{amsthm}
\usepackage{epsfig}
\usepackage[english]{babel}
\usepackage{hyperref}
\usepackage[normalem]{ulem}
\usepackage{mathrsfs}
\usepackage{tikz}
\usetikzlibrary{decorations.markings,backgrounds}
\usetikzlibrary{arrows.meta}
\usepackage{pgfplots}
\usepackage{subfigure}
\usepackage{caption}
\usepackage{bbm}
\usepackage{enumitem}
\usepackage{amsaddr}

\usepackage{stmaryrd}

\usepackage{float}
\usepackage{pst-plot}

\hypersetup{
	colorlinks,
	linkcolor={red!80!black},
	citecolor={blue!50!black},
	urlcolor={blue!80!black}
}
 \usepackage{diagbox}

\usepackage[a4paper,top=3.5 cm,bottom=3 cm,left=2.5 cm,right=2.5 cm]{geometry}

\theoremstyle{plain}
\newtheorem{lemma}{Lemma}[section]
\newtheorem{proposition}[lemma]{Proposition}
\newtheorem{theorem}[lemma]{Theorem}
\newtheorem{corollary}[lemma]{Corollary}

\theoremstyle{definition}
\newtheorem{definition}[lemma]{Definition}

\theoremstyle{remark}
\newtheorem{remark}[lemma]{Remark}

\numberwithin{equation}{section}

\def\to{\rightarrow}

\usepackage{dsfont}

\newcommand{\D}{\mathcal{D}}

\newcommand{\nchi}{{\raise.3ex\hbox{$\chi$}}}

\def\XXint#1#2#3{{\setbox0=\hbox{$#1{#2#3}{\int}$ }
		\vcen{\hbox{$#2#3$ }}\kern-.6\wd0}}

\newcommand{\cD}{\mathcal{D}}

\newcommand{\cI}{\mathcal{I}}

\newcommand{\cL}{\mathcal{L}}

\newcommand{\C}{\mathbb{C}}

\newcommand{\R}{\mathbb{R}}

\newcommand{\Z}{\mathbb{Z}}

\usepackage{mathtools}

\newcommand{\dd}{\, \mathrm{d}}

\makeatletter
\newcommand{\justified}{%
	\rightskip\z@skip%
	\leftskip\z@skip}
\makeatother
\usepackage{amsfonts}
\usepackage{aurical}

\newcommand\restr[2]{{
		\left.\kern-\nulldelimiterspace 
		#1 
		\vphantom{\big|} 
		\right|_{#2} 
}}

\renewcommand{\d}{{\rm d}}
\newcommand{\eps}{{\varepsilon}}

\newcommand{\dxi}{{\, \mathrm{d}\xi}}

\newcommand{\var}{\varepsilon}

\newcommand{\IR}{\mathbb{R}}

\DeclareFontFamily{U}{mathx}{\hyphenchar\font45}
\DeclareFontShape{U}{mathx}{m}{n}{<-> mathx10}{}
\DeclareSymbolFont{mathx}{U}{mathx}{m}{n}
\DeclareMathAccent{\widebar}{0}{mathx}{"73}

\newcommand{\ID}{\mathcal{D}}

\newcommand{\IF}{\mathcal{F}}

\newcommand{\IC}{\mathcal{C}}

\def\d{\partial}
\usepackage{natbib}

\allowdisplaybreaks
\usepackage[colorinlistoftodos]{todonotes}
\usepackage{url}

\usepackage{graphicx,tikz}

\title[Finite difference method for partially dissipative hyperbolic systems]{\vspace*{-2cm}Asymptotic-preserving finite difference method for partially dissipative hyperbolic systems}

\author[T. Crin-Barat $\&$ D. Manea]{Timothée Crin-Barat$^{1}$ and Dragos Manea$^{2,3,\dagger}$}
\address{\begin{center}$^1$Chair for Dynamics, Control, Machine Learning and Numerics, Friedrich-Alexander-Universität Erlangen-Nürnberg, 91058 Erlangen, Germany.\end{center}
\mbox{} \vspace{-0.5cm}
\begin{center}
$^2$Institute of Mathematics “Simion Stoilow” of the Romanian Academy,
21 Calea Grivitei Street, 010702 Bucharest, Romania.\end{center}
\mbox{} \vspace{-0.5cm}
\begin{center}$^3$The Research Institute of the University of
Bucharest - ICUB, University of Bucharest, 90-92 Sos. Panduri, 5th District, Bucharest, Romania.\end{center}}

\thanks{$^\dagger$Corresponding author: dmanea28@gmail.com}


\usepackage{graphicx}

\begin{document}
\begin{abstract}
In this paper, we analyze the preservation of asymptotic properties of partially dissipative hyperbolic systems when switching to a discrete setting. We prove that one of the simplest consistent and unconditionally stable numerical methods -- the central finite difference scheme -- preserves both the asymptotic behaviour and the parabolic relaxation limit of one-dimensional partially dissipative hyperbolic systems which satisfy the Kalman rank condition.

The large time asymptotic-preserving property is achieved by conceiving time-weighted perturbed energy functionals in the spirit of the hypocoercivity theory. For the relaxation-preserving property, drawing inspiration from the observation that solutions in the continuous case exhibit distinct behaviours in low and high frequencies, we introduce a novel discrete Littlewood-Paley theory tailored to the central finite difference scheme. This allows us to prove Bernstein-type estimates for discrete differential operators and leads to a new relaxation result: the strong convergence of the discrete linearized compressible Euler system with damping towards the discrete heat equation, uniformly with respect to the mesh parameter. 
\end{abstract}

\maketitle

\vspace{-0.7cm}

\tableofcontents

\section{Introduction and informal results}
Extensive literature exists on the analysis of partially dissipative hyperbolic models, particularly focusing on their asymptotic behaviour and singular limits using a combination of Fourier and hypocoercivity techniques.
While in the continuous setting there is growing progress in understanding these phenomena, a persistent challenge arises when transitioning to a numerical context and seeking to preserve such properties in a grid-uniform manner.

In this context, our research contains both theoretical and experimental evidence for the fact that hypocoercivity and relaxation properties inherent to partially dissipative hyperbolic systems can be effectively captured by one of the simplest and unconditionally stable numerical schemes: the central finite difference method.

\subsection{Partially dissipative systems -- propagation of damping through hyperbolic dynamics}

We are concerned with the numerical analysis of partially dissipative hyperbolic systems of the form
\begin{equation}\label{SystGen}
\begin{aligned}
    &\d_t U + A\d_xU =-BU, & (t,x)\in (0,\infty)\times \R,
\end{aligned}
\end{equation}
where $U=U(t,x)\in \mathbb{R}^{N}$ ($N\geq2$) is the vector-valued unknown and $A,B$ are symmetric $N\times N$ matrices. We assume that \eqref{SystGen} has a \textit{partially dissipative structure}: The matrix $B$ takes the form 
\begin{align} \label{BD}
B=\begin{pmatrix} 0 & 0 \\ 0 & \widetilde{B}
    \end{pmatrix},
\end{align}
    where $\widetilde{B}$ is a positive definite symmetric $N_2\times N_2$  matrix $(1\leq N_{2}<N)$. 
    Under these conditions, $\widetilde{B}$ satisfies the \textit{strong dissipativity condition}: there exists a constant $\lambda>0$ such that, for all $X\in\mathbb{R}^{N_{2}}$, 
    \begin{align}\label{Strong Dissipativity}
        \langle \widetilde{B}X,X\rangle\geq \lambda |X|^2,
    \end{align}
    where $\langle\:,\rangle $ denotes the inner product on $\mathbb{R}^N$.  Based on this definition, we decompose the solution as $U=(U_1,U_2)$ where $U_1\in \R^{N-N_2}$ corresponds to the conserved components and $U_2\in \R^{N_2}$ to the dissipated ones. In general, the $L^2$-stability of these systems is unclear, as the dissipative operator $\widetilde{B}$ only acts on $U_2$.

    \cite{SK} observed that if the eigenvectors of $A$ avoid the kernel of the dissipative matrix $B$ (this requirement is called the \textit{SK condition}), then the solutions are stable in $L^2$. More recently, \cite{BZ} established a link between the SK condition, control theory and the theory of hypocoercivity \cite[]{Villani}. In particular, they constructed perturbed energy functionals permitting to recover the asymptotic behaviour of the solutions of \eqref{SystGen} under the Kalman rank condition:
\begin{definition}
A pair of matrices $(A,B)$ verifies the \textit{Kalman} rank condition if
         \begin{align} \label{K}\tag{K}
 \text{the matrix    }\:\mathcal{K}(A,B):=(B|AB|\ldots|A^{N-1}B)\quad \text{has full rank } N.
 \end{align}
 \end{definition}
In practice, this condition means that the partially dissipative effects of $B$ can be propagated to the other components through the hyperbolic dynamics of the system. In \cite{BZ} and numerous references on the topic (e.g. \cite{Yong,BHN,HanouzetNatalini,CBD2,CBD1}), the authors rely on Fourier techniques to derive the large time behaviour of the solutions. Recently, \cite{CBSZ} developed a Fourier-free method prone to tackle situations where the Fourier transform cannot be employed such as e.g. bounded domains, time and space-dependent matrices or Riemannian manifolds. Their method leads to the following \textit{natural} time-decay estimates for the solution of \eqref{SystGen}.
\begin{theorem}[{\cite[Theorem 2.1]{CBSZ}}]\label{DecayThm1}
 Let $U_{0}\in H^1(\R)$, $A$ and $B$ be symmetric $N\times N$ matrices, with $B$ as in \eqref{BD}, satisfying the Kalman rank condition. Then, for all $t>0$, the solution $U$ of  \eqref{SystGen} with the initial datum $U_{0}$ satisfies
 \begin{equation}
\label{DecayBZ}
\| U_2(t)\|_{L^2(\R)}+\|\d_x U(t)\|_{L^2(\R)} \leq C(1+t)^{-\frac{1}{2}}\|U_0\|_{H^1(\R)},
\end{equation}
where $C>0$ is a constant independent of time and $U_0$.
\end{theorem}
This proposition highlights the hypocoercive nature inherent to partially dissipative systems. Although the damping term does not directly influence every component of the system, the whole solution decays in time due to the cross-influence between the matrices $A$ and $B$ resulting from the Kalman condition \eqref{K}. The decay rate resembles that of the heat equation for $L^2$ data, but, in this hyperbolic framework, we need to assume additionally that the initial datum is in $H^1$, due to the lack of parabolic regularising effects. Moreover, in \cite{CBSZ}, it is shown that the decay \eqref{DecayBZ} is optimal for $H^1$ initial data. The lack of exponential decay essentially comes from the fact that the hyperbolic and dissipative operators in the system are of different orders.
\begin{remark}
    A classical system fitting the description \eqref{SystGen}-\eqref{BD} and verifying the Kalman rank condition \eqref{K} is the compressible Euler equations with damping, linearized around $(\bar{\rho},\bar{u})=(\bar{\rho},0)$, with $\bar{\rho}>0$. The linearisation takes the following form
\begin{equation}\label{Euler}
\left\{
\begin{aligned}
&\partial_{t}\rho+\d_x u=0, \\
&\partial_{t} u+\d_x \rho+u=0,
\end{aligned}
\right.\quad  (t,x)\in (0,\infty)\times \R,
\end{equation}
where $\rho=\rho(x,t)\geq 0$ denotes the fluid density and  $u=u(x,t)\in\mathbb{R}$ stands for the fluid velocity.
\end{remark}

\subsection{Asymptotic-preserving schemes for partially dissipative systems}

The first purpose of this paper is to prove that one of the simplest structure-preserving, i.e. consistent and stable, numerical schemes for \eqref{SystGen}, namely the one based on the central finite difference discrete operator on a uniform $h$-sized grid:
\begin{align}
\label{eq:FDcentral*}
(\ID_h U)_n&=\frac{U_{n+1}-U_{n-1}}{2h},\, n\in \Z,
\end{align}
preserves the large-time asymptotics derived in Theorem \ref{DecayThm1}. 
The choice of this particular scheme is justified by its unconditional stability for hyperbolic systems, see Proposition \ref{stabilityForSchemes} and Section \ref{sec:FD} for additional details regarding the discrete framework employed.

\smallbreak
Next, we present an informal version of our asymptotic behaviour result:
\smallbreak

\textit{The central finite difference scheme is asymptotic-preserving for the system \eqref{SystGen} in the sense that we recover the time-decay \eqref{DecayBZ} for the semi-discretized version of \eqref{SystGen}, uniformly with respect to the mesh-size parameter, when the Kalman rank condition holds.}
\smallbreak
The complete version of this result can be found in  Theorem \ref{thm:NumHyp}.

\subsection{Relaxation-preserving scheme}
\label{sec:intro-relaxation-preserving}
Up to this point, we have at hand a structure-preserving and asymptotic-preserving numerical scheme for the system \eqref{SystGen}. One of the natural further steps is to analyse whether this scheme behaves well with respect to another type of asymptotics: singular perturbations. More precisely,  \cite{YongIllPrep} showed that the system \eqref{SystGen} can be relaxed to a parabolic one in a diffusive scaling.
We will prove that, under certain regularity conditions on the initial data, the same relaxation behaviour is observed in the discrete setting, uniformly with respect to the grid width $h$.

We investigate such approximations in a simple, yet illustrative case: the discrete version of the linear compressible Euler system \eqref{Euler} with relaxation
\begin{equation}\label{mainR:inEuler} 
\left\{\begin{aligned}
   & \d_t \rho^\varepsilon + \ID_hu^\varepsilon=0,
    \\ & \varepsilon^2\d_t u^\varepsilon + \ID_h\rho^\varepsilon +u^\varepsilon=0,
\end{aligned}\right.
\end{equation}
where $\rho^\varepsilon, u^\varepsilon: (0,\infty)\times \Z\to \IR$  and $\varepsilon>0$ is the relaxation parameter. As $\varepsilon\to0$, the solutions of \eqref{mainR:inEuler} converge, at least formally, to the solutions of the discrete heat equation:
\begin{equation}\label{mainR:heat}
\left\{\begin{aligned}
   & \d_t \rho -\ID_h^2\rho=0,\\
   & u=-\ID_h \rho,
   \end{aligned}\right.
\end{equation}
where the second equation corresponds to the discrete Darcy law.

\input{plot-sinx-x}

In the continuous context, \cite{OriveZuazua06} and \cite{CBD3} showed that it is essential to analyze separately the low and high frequencies of the solutions to derive strong convergence results for such relaxation procedures. For instance, in \cite{CBD3}, the authors
 introduce a hybrid Littlewood-Paley decomposition to justify the strong convergence of the nonlinear compressible Euler system with damping toward the porous media equation in any dimension. In the present paper, to justify the relaxation-preserving property of the numerical scheme, we introduce a novel and numerically suited Littlewood-Paley decomposition. In this regard, the main challenge that arises is that the Fourier symbol of the discrete operator, which is:
\begin{equation}
\label{eq:intro-symbol}
\widehat{(\ID_h v)}(\xi)=i\frac{\sin(\xi h)}{h} \hat{v}(\xi),\quad \xi \in \left[-\frac{\pi}{h},\frac{\pi}{h}\right],
\end{equation}
becomes very small at high frequencies $\xi \sim \pm \frac{\pi}{h}$. Therefore, we are not able to uniformly approximate $\frac{\sin(\xi h)}{h}$ by $\xi$, since, in the high-frequency regime, $\left|\frac{\sin(\xi h)}{\xi h}\right|\ll 1$ (see Figure \ref{fig:plot-sinx-x}). To tackle this issue, we develop a non-standard dyadic decomposition tailored to the central finite difference operator $\D_h$. More precisely, whereas in the continuous Littlewood-Paley theory, the dyadic decomposition of the frequency domain is done in logarithmically equidistant annuli:
\begin{equation}\label{eq:intro-continuous-freq-band}
|\xi|\in \left[\frac{3}{4}2^{j},\frac{4}{3}2^{j+1}\right],
\end{equation}
in our case, we work with a numerically adapted dyadic decomposition based on non-uniform annuli of the form:
\begin{equation}\label{eq:intro-discrete-freq-band}
F_h(j):=\left\{\xi \in \left[-\frac{\pi}{h},\frac{\pi}{h}\right]:\left|\frac{\sin(\xi h)}{h}\right|\in \left[\frac{3}{4}2^{j},\frac{4}{3}2^{j+1}\right]\right\}.
\end{equation}
See Figure \ref{fig:dyadic-decompositions} for a comparison between the decomposition intervals in \eqref{eq:intro-continuous-freq-band} and \eqref{eq:intro-discrete-freq-band}.  On that figure, we remark that the numerically adapted dyadic decomposition showcases a pseudo-low-frequency regime near the boundary of the frequency domain (i.e. for $\xi \sim \pm \frac{\pi}{h}$), which will need special treatment in our analysis (for example, in the proof of Proposition \ref{prop:embed}). Moreover, we observe that the discrete high-frequency regime does not align with the continuous one, since the former is bounded for a fixed mesh width $h$. Nevertheless, since, as $h$ approaches zero, we recover the full high-frequency range, we employ a full spectrum analysis in Section \ref{sec:unif-in-epsilon-estimates} to obtain convergence results independent of $h$.

\begin{figure}[H]
    \centering
    \includegraphics[scale=0.24]{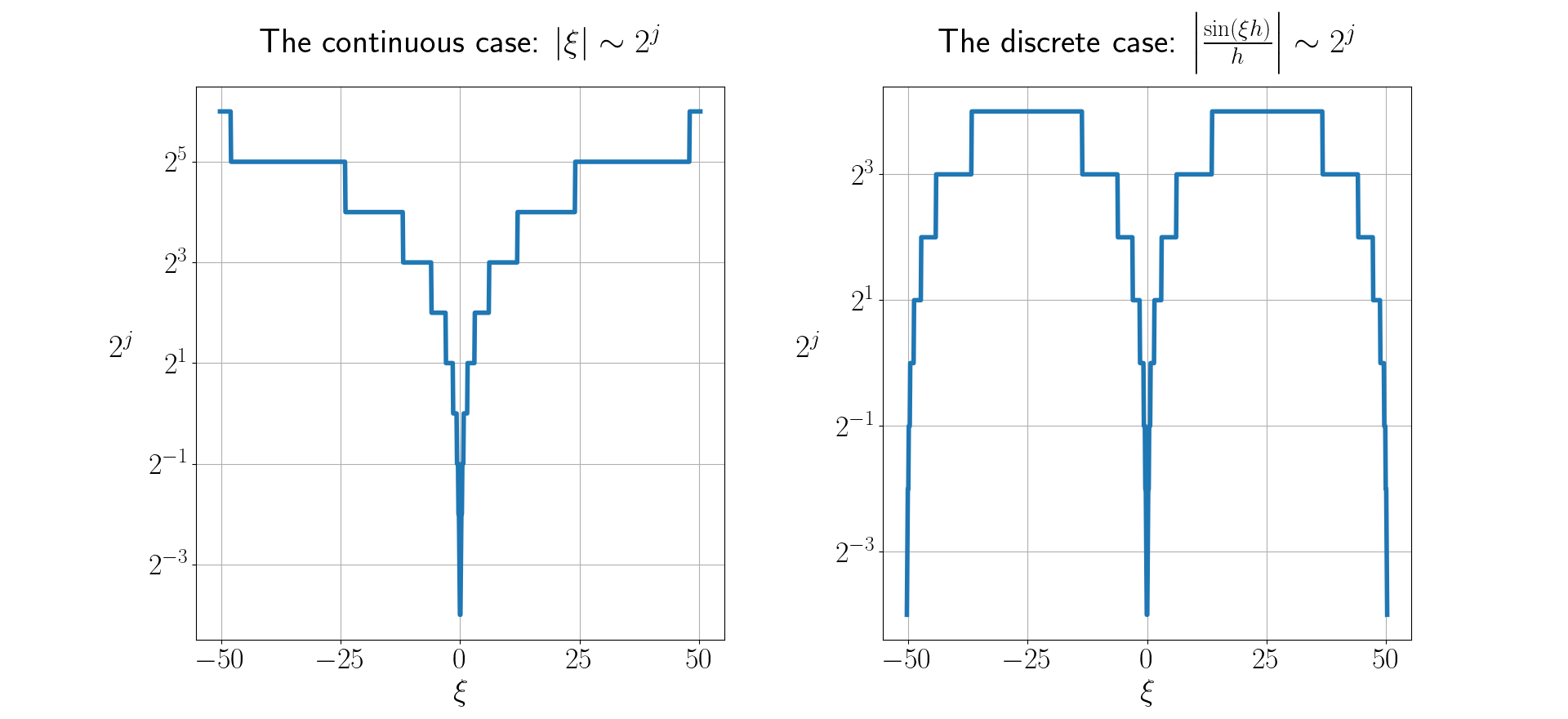}
    \caption{The decomposition of the frequency space in the continuous case \eqref{eq:intro-continuous-freq-band} (left) and in the discrete setting \eqref{eq:intro-discrete-freq-band} (right), for $h=2^{-4}$.}
    \label{fig:dyadic-decompositions}
\end{figure}

We also note that our approach differs from the discrete Littlewood-Paley decomposition outlined in \cite{HongYang2019}, which utilizes the dyadic decomposition \eqref{eq:intro-continuous-freq-band}. One significant improvement is that the numerically adapted decomposition \eqref{eq:intro-discrete-freq-band} we employ enables us to establish Bernstein-type estimates for the operator $\D_h$ in the following sense:
\[c\, 2^j \|\delta_h^j v\|_{l^2_h} \leq \|\ID_h \delta_h^j v\|_{l^2_h} \leq C\, 2^j \|\delta_h^j v\|_{l^2_h},\]
where $\delta_h^j v$ is the localization of the grid function $v$ to the frequency band \eqref{eq:intro-discrete-freq-band} -- see Definition \ref{def:localization} -- and $c,C$ are universal constants.

This leads to the following definition of homogeneous discrete Besov semi-norms: for a regularity index $s\in \R$, we define:
  \begin{equation}
    \label{eq:eq:BesovNormLFHF0}
    \|v\|_{\dot{B}_{h}^s}:=\sum_{j\in \Z} 2^{js} \|\delta_h^j v\|_{l_h^2}.
    \end{equation}
Within this framework, to split our analysis into low and high frequencies, we simply need to apply the frequency-localization linear operator $\delta^j_h$ to the system and study the solution separately for $j\ll 0$ and $j\gg0$, respectively. Furthermore, the spaces defined in \eqref{eq:eq:BesovNormLFHF0} are related to classical discrete functional spaces by the following embedding (see Proposition \ref{prop:embed}):
   \begin{equation}
   \dot{B}^{s}_h\hookrightarrow \dot{h}^s_h \quad \text{and} \quad \dot{B}^{\frac 12}_h\hookrightarrow \ell^\infty_h,
   \end{equation}
    where the homogeneous and inhomogeneous discrete Sobolev norms are defined as:
    \begin{align}   \|v\|_{\dot{h}^s_h}^2:=\|\mathcal{D}^s_hv\|_{\ell^2_h}:=\left\|\hat{v}(\xi)\left|\frac{\sin(\xi h)}{h}\right|^s\right\|_{L^2\left(-\frac{\pi}{h},\frac{\pi}{h}\right)} \quad \textrm{and} \quad \|v\|_{h^s_h}^2:=\|v\|_{\ell^2_h}^2+\|\mathcal{D}^s_h v\|_{\ell^2_h}^2.
    \end{align} 
We refer to Section \ref{sec:discrete_framework} for more information about the functional framework we use.

\medbreak
Next, we present an informal version of our relaxation-preserving result:
\smallbreak
\textit{The central finite difference scheme is relaxation-preserving for the system \eqref{mainR:inEuler} in the sense that, for all times $T>0$, the solutions of \eqref{mainR:inEuler} converge to the solutions of \eqref{mainR:heat} in the supremum and Sobolev norms at the rate $\mathcal{O}(\varepsilon^2)$, uniformly with respect to the mesh-size parameters.}

\smallbreak
The rigorous form of this result can be found in Theorem \ref{thm:NumRelax} and Corollary \ref{cor:lInftyRelax}.

\section{Main results}
\subsection{Asymptotic-preserving property of the central finite difference scheme}
In this section, we establish the counterpart of \eqref{DecayBZ} for discrete hyperbolic systems, demonstrating the preservation of the hypocoercivity property when transitioning to a numerical context. To be more specific, we derive time-decay rates for a space-discretization of \eqref{SystGen} on a uniform grid of width $h$. We consider $U:(0,\infty)\times \Z \to \IR^N$ that satisfies the following semi-discrete equation:
\begin{align}\label{mainr:Systn}
   \partial_t U(t)+ A(\ID_h U(t))&=-B U(t).
\end{align}
The next result establishes time-decay estimates for \eqref{mainr:Systn}, uniformly with respect to the mesh width.
\begin{theorem}[Numerical hypocoercivity for hyperbolic systems]\label{thm:NumHyp}
Let $U_{0}\in h_h^1$ , $A$ and $B$ be symmetric $N\times N$ matrices with $B$ as in \eqref{BD} and such that $(A,B)$ satisfies the Kalman rank condition \eqref{K}. Then, for all $t>0$, the solution $U $ of  \eqref{mainr:Systn} with the initial datum $U_{0}$ satisfies
 \begin{equation}
\label{e:decayThm1}
\| U _{2}(t)\|_{\ell^2_h}+\|\ID_hU (t)\|_{\ell^2_h} \leq C(1+t)^{-\frac{1}{2}}\|U_{0}\|_{h^1_h},
\end{equation}
where $C>0$ is a constant independent of the mesh-size parameter $h$, the time $t$ and $U_0$.
\end{theorem}

\begin{remark} The decay rate obtained in \eqref{e:decayThm1} is sharp with respect to the one derived in the continuous case in \cite{CBSZ}. Applying Theorem \ref{thm:NumHyp} to the linearized compressible Euler system \eqref{mainR:inEuler}, we obtain, for all $t>0$, 
     \begin{equation}
\label{e:decayThm12}
\|u(t)\|_{\ell^2_h}+\|\ID_h(\rho,u) (t)\|_{\ell^2_h} \leq C(1+t)^{-\frac{1}{2}}\|(\rho_0,u_0)\|_{h^1_h}.
\end{equation}
The sharpness of the decay rate is also validated by the simulations in Section  \ref{sec:hypo-experiments}.
    \end{remark}
    \bigbreak
\noindent\textbf{Strategy of proof and comparison with the literature.}
To establish Theorem \ref{thm:NumHyp}, we construct time-weighted Lyapunov functionals inspired by the recent work 
\cite{CBSZ}. Their approach, employing various tools to analyze partially dissipative systems without relying on the Fourier transform, broadens the scope of applicability beyond standard methods, such as \cite{SK, BZ, BHN} and \cite{CBD2}. The construction of the time-weighted Lyapunov functionals is influenced by
the works \cite{Herau} and \cite{HerauNier} on the asymptotic behaviour of hypocoercive kinetic models and \cite{BZ} concerning the hypocoercivity phenomenon for hyperbolic systems.  In the present paper, the Lyapunov functionals we use closely resemble the one in \cite{CBSZ} and are tailored for the central finite difference approximation of the partially dissipative system \eqref{SystGen}. Differentiating these functionals with respect to time and employing the Kalman rank condition \eqref{K}, we derive the desired time-decay rates.
\smallbreak
Regarding the discrete asymptotic stability of partially dissipative systems, numerous studies are dedicated to formulating asymptotic-preserving numerical schemes for hypocoercive phenomena. Closely connected to our work, we highlight the contributions of 
\cite{PorrettaZuazua2016} and 
\cite{Georgoulis21}, where time-decay estimates are derived for discretized versions of the two-dimensional Kolmogorov equation, employing finite difference and finite element schemes, respectively. In the broader context of kinetics models, particularly emphasizing the Fokker-Planck equation, we refer to \cite{BlausteinFilbet24,DujardinHérauLafitte20,FilbetRodrigues17,FilbetRodriguesZakerzadeh21,FosterLohéacTran17,Bessemoulin-ChatardHerdaRey20} and references therein. Concerning the issue of stability of finite difference schemes for hyperbolic systems, we refer to \cite{Trefethen84,Coulombel09,CoulombelFaye23,MaHuangYong2023}.

\subsection{Strong relaxation limit in the semi-discrete setting}
\subsubsection{A new frequency-based discrete framework}
 The proof of our second main result -- the discrete relaxation limit -- is inspired by \cite{CBD3} pertaining to the continuous setting. In this reference, it is shown that the solutions of the nonlinear compressible Euler system converge strongly, in suitable norms, as the relaxation parameter $\varepsilon$ approaches zero, to the solutions of the porous media equation. There, the authors use a frequency-splitting method and treat the low and high frequencies in two different manners. Importantly, in their approach, the threshold between low and high frequencies is located at $1/\varepsilon$, which implies that the high-frequency regime disappears in the limit $\varepsilon\to0$. 

Drawing upon these insights, to obtain new results related to hyperbolic relaxation procedures for semi-discrete hyperbolic systems, we employ the novel construction of Besov spaces roughly described in Section \ref{sec:intro-relaxation-preserving} and rigorously introduced in Section \ref{sec:discrete-besov}.
Another key concept, inspired from \cite[Section 10.1]{Strikwerda2004}, is the truncation operator: $\mathcal{T}_h:L^2(\IR)\to \ell^2_h$,
\begin{equation}
    \label{def:truncation}
    (\mathcal{T}_hv)_n = \frac{1}{\sqrt{2\pi}}\int_{-\frac{\pi}{h}}^{\frac{\pi}{h}} e^{i\xi n h} \hat{v}(\xi) \dxi,
\end{equation}
where $\hat{v}$ is the continuous Fourier transform of $v\in L^2(\IR)$. It is notable that the discrete Fourier transform $\widehat{T_hv}(\xi)$, as introduced in Definition \ref{def:discreteFourier}, coincides with the continuous Fourier transform $\hat{v}(\xi)$ for any $\xi \in \left[-\frac{\pi}{h},\frac{\pi}{h}\right]$. Thus, the purpose of this truncation operator is to transfer functions defined on the real line to a grid of width $h$, while preserving the Fourier transform. For a comprehensive understanding of the suitability of this operator in accurately projecting functions onto the $h$-grid, interested readers can refer to \cite[Theorems 10.1.3 and 10.1.4]{Strikwerda2004}. The forthcoming result concerning the truncation operator will ultimately ensure that the constants in our relaxation result are uniform with respect to the grid width $h$.

\begin{theorem}[Uniform Besov estimates with respect to the grid width]
\label{thm:unif_Besov_estimate}
For every $s'>0$ and $s\in (0,s')$ there exists a constant $C_{s',s}>0$ such that, for every $h>0$ and every $u\in H^{s}(\IR)$, we have
\begin{equation}
    \label{eq:estimateBesovUnif_h}
    \|\mathcal{T}_h v\|_{\dot{B}^{s}_{h}}\leq C_{s',s}\|v\|_{H^{s'}(\IR)}.
\end{equation}
\end{theorem}

Furthermore, the discrete Besov spaces that we employ can be embedded in the classical Sobolev and $\ell^\infty_h$ spaces as follows.
\begin{proposition}\label{prop:embed}
For every $s\in \R$, there exists a constant $C_s>0$ depending only on $s$ such that, for every $v\in \dot{B}_{h}^{s}$, the following inequality holds true:
\begin{equation}
    \|v\|_{\dot{h}_h^s} \leq C_s \|v\|_{\dot{B}_{h}^{s}}.
\end{equation}
Moreover, in the particular case $s=\frac{1}{2}$, the discrete Besov space is continuously embedded in $l_h^\infty$:
\begin{equation}
\label{eq:l-infty-embed}
\|v\|_{l^\infty_h} \leq C \|v\|_{\dot{B}_{h}^{\frac{1}{2}}},
\end{equation}
where $C$ is a universal constant.
\end{proposition}

\subsubsection{Discrete strong relaxation of the incompressible linearized Euler system}

Before presenting our relaxation result, we introduce the concept of $(s,h)$-truncated initial data, which allows us to essentially use a single initial data in \eqref{mainR:inEuler} and \eqref{mainR:heat}, regardless of the grid width $h$. 
\begin{definition}[$(s',h)$-truncation]
\label{def:h-truncation}
Let $\varepsilon>0$ and $s'>2$. We say the initial data $(\rho_0^*,u_0^*)$ of \eqref{mainR:inEuler} and $\rho_0$ of \eqref{mainR:heat} are  $(s',h)$\textit{-truncated of order} $\varepsilon^2$ if there exist $\widetilde{\rho_0}\in H^{s'-2}(\IR)$ and $\widetilde{\rho_0^*},\widetilde{u_0^*}\in H^{s'}(\IR)$ such that the following relations hold:
\begin{enumerate}[label={(\roman*)}]
    \item $\|\widetilde{\rho_0}-\widetilde{\rho_0^*}\|_{H^{s'-2}(\IR)}<\varepsilon^2$,
    \item $\rho_0= \mathcal{T}_h(\widetilde{\rho_0})$,
    \item $\rho_0^*= \mathcal{T}_h(\widetilde{\rho_0^*})$,
    \item $u_0^*= \mathcal{T}_h(\widetilde{u_0^*})$.
\end{enumerate}
\end{definition}

The above definition ensures that the initial data for the discrete Cauchy problems \eqref{mainR:inEuler} and \eqref{mainR:heat} are obtained by truncating sufficiently regular functions defined on the real line.
\smallbreak
In order to state our result regarding the relaxation limit from \eqref{mainR:inEuler} to \eqref{mainR:heat}, we need the following notation:
for any Banach space $X$, $T>0$ and $p\in [1,\infty]$, we denote by $L^{p}_T(X)$ the set of measurable functions $g:[0,T]\rightarrow X$ such that $t\mapsto \|g(t)\|_{X}$ is in $L^{p}(0,T)$.

\begin{theorem}[Numerical relaxation limit]\label{thm:NumRelax}
Let $\varepsilon,h>0$, $s'>2$ and $s\in (2,s')$. We assume that $(\rho_0^*,u_0^*,\rho_0)$ are $(s',h)$-truncated of order $\varepsilon^2$ and we denote $(\rho^\varepsilon,u^\varepsilon)$ and $\rho$ the corresponding solution of \eqref{mainR:inEuler} and \eqref{mainR:heat}, respectively. Then, for every time $T>0$, $\rho^\varepsilon$ converges strongly toward $\rho$ in the following sense:
 \begin{align}
     \label{StrongRelax}
        \|(\rho^\varepsilon-\rho)(T)\|_{\dot{B}^{s-2}_{h}}+\|\rho^\varepsilon-\rho\|_{L^1_T(\dot{B}^{s}_{h})}+\|\ID_h\rho^\varepsilon+u^\varepsilon\|_{L^1_T(\dot{B}^{s-1}_{h})}\leq C\varepsilon^2,
 \end{align}
where $C=C_1\left(1+\|(\widetilde{\rho_0^*},\varepsilon \widetilde{u_0^*})\|_{H^{s'}}+\|\widetilde{u_0^*}\|_{H^{s'-1}}\right)$, with $C_1>0$ a constant independent of $h$, $\varepsilon$, $T$ and the initial data.
\end{theorem}

Combining Proposition \ref{prop:embed} and Theorem \ref{thm:NumRelax}, we obtain the strong convergence in $\dot{h}_h^{s}$ and $\ell^\infty_h$, uniformly in $h$ and $T$.
 \begin{corollary}\label{cor:lInftyRelax}
Let all the assumptions of Theorem \ref{thm:NumRelax} be in force. The following statements hold true:
\begin{enumerate}[label={(\roman*)}]
    \item \label{cor:relax-hs} For every $T>0$, the solutions of \eqref{mainR:inEuler} converge strongly, as $\varepsilon\to0$, to the solutions of \eqref{mainR:heat}, in $L^\infty_T(\dot{h}_h^{s-2})$ and $L^1_T(\dot{h}_h^{s})$ at the rate $\mathcal{O}(\varepsilon^2)$, uniformly in $h$.  Moreover, the quantity $\ID_h\rho^\varepsilon+u^\varepsilon$ associated to the Darcy law converges to $0$ in $L^1_T(\dot{h}^{s-1})$ at the rate $\mathcal{O}(\eps^2)$, uniformly in $h$.
\item \label{cor:relax-linfty} For every $T>0$ and for $s'>5/2$, the solution $\rho^\eps$ of \eqref{mainR:inEuler} converge strongly, as $\varepsilon\to0$, to the solution  $\rho$ of \eqref{mainR:heat} in $L^\infty_T(\ell^\infty_h)$ and $L_T^1(l_h^\infty)$ at the rate $\mathcal{O}(\varepsilon^2)$, uniformly in $h$. Furthermore, the quantity $\ID_h\rho^\varepsilon+u^\varepsilon$ converges to $0$ in $L^1_T(\ell^\infty_h)$ at the rate $\mathcal{O}(\eps^2)$, uniformly in $h$.
\end{enumerate}
\end{corollary}

\begin{remark}
   The convergence rate we obtain is one order higher than the one obtained in \cite{CBD3} for the nonlinear compressible Euler system with damping. This is at the cost of stronger regularity requirements for the initial data. The result we obtain is sharp, since it is consistent with the rate observed in the numerical simulations in Section \ref{sec:relax-experiments}.
\end{remark}
  
\begin{remark}
In Section \ref{sec:ext}, we broaden the scope of Theorem \ref{thm:NumRelax} to encompass general hyperbolic systems satisfying the Kalman rank condition. The analysis for this extension closely follows the methodology outlined in \cite{DanchinEMS}, assuming supplementary conditions on the matrices $A$ and $B$.
\end{remark}
\noindent\textbf{Comments and comparison with the literature.}
In our approach, a distinctive advantage of our discrete Littlewood-Paley decomposition, in contrast to existing literature (such as the work \cite{HongYang2019} on Strichartz estimates for discrete Schrödinger and Klein-Gordon equations), lies in the adaptation of the localization annuli to the precise form of the discrete differential operator $\ID_h$. This adaptation allows us to justify Bernstein-type estimates necessary for proving the relaxation property.

The justification of our results differs from previous endeavours related to similar hyperbolic approximation procedures, e.g. \cite{BoscarinoRusso09,ShiJinReview,HuShu2023uniform,BoscarinoRusso24} where implicit–explicit (IMEX) Runge–Kutta schemes are used. Here, we establish the asymptotic-preserving property of the central finite difference scheme within a refined frequency-based functional framework, strategically constructed to approach stiff relaxation procedures for hyperbolic systems. In a related context, we highlight \cite{Degond13,DimarcoPareschi14,Jin22, LemouMieussens, Jin10, JinPareschiToscani00, Bessemoulin-ChatardHerdaRey20, GoudonJinLiuYan13,BlausteinFilbet24} where authors delved into the diffusive limit of kinetics and hyperbolic models. In particular, in the recent work \cite{BlausteinFilbet24},
the authors
craft a discrete framework for studying the Vlasov-Poisson-Fokker-Planck system, first rewriting the equations as a partially dissipative hyperbolic system with stiff relaxation terms, using Hermite polynomials in terms of the velocity. Then, in line with the continuous approach by
\cite{DMS15}, they justify the relaxation limit of such hyperbolic systems (which shares similarities with the one studied in the present paper), revealing the diffusion limit at the discrete level of the kinetic model.
A crucial difference in our current scenario is that we tackle the full-space case, as opposed to the torus setting. In the full-space case, there is a lack of a spectral gap in low frequencies due to the absence of a Poincaré-type inequality, thus leading to a dichotomous behaviour in low and high frequencies which, in turn, requires the development of a functional framework tailored to deal with this polarity.

\section{Presentation of the discrete framework}
\label{sec:discrete_framework}

\subsection{Discrete Fourier transform}
Within this section, we introduce the discrete one-dimensional Fourier transform and revisit some fundamental properties such as invertibility and Parseval's equality. Subsequently, we use these properties to study the solutions of discrete hyperbolic systems.

The following definition is essentially taken from \cite[Section 2.2]{Trefethen1994}.
\begin{definition}[Discrete Fourier transform]
\label{def:discreteFourier}
We consider a bilateral infinite real sequence $(v_n)_{n\in \Z}$ and a grid width $h>0$. Assume that $v\in l^2_h$, that is:
\begin{equation}
\label{eq:defl2h}
    \|v\|_{l^2_h}^2:=h \sum_{n\in \Z} v_n^2<\infty.
\end{equation}
The discrete Fourier transform of $v$ is defined as $\hat{v}:\left[-\frac{\pi}{h},\frac{\pi}{h}\right]\to \IR$,
\[\hat{v}(\xi):=\frac{h}{\sqrt{2\pi}}\sum_{n\in \Z} e^{-i\xi nh} v_n.\]
The inverse Fourier transform $\IF^{-1}:L^2\left(\left[-\frac{\pi}{h},\frac{\pi}{h}\right]\right)\to l^2_h$ has the following form:
\begin{equation}
\label{eq:inverseDiscreteFourier}
\left(\IF^{-1}(g)\right)_n=\frac{1}{\sqrt{2\pi}} \int_{-\frac{\pi}{h}}^{\frac{\pi}{h}} e^{i\xi nh} g(\xi) \dxi.
\end{equation}
\end{definition}

The next proposition, taken from \cite[Theorem 2.5]{Trefethen1994}, summarizes some basic properties of the discrete Fourier transform.
\begin{proposition}
    Let $v\in l^2_h$. The following properties hold:
    \begin{enumerate}
        \item $\hat{v}\in L^2\left(\left[-\frac{\pi}{h},\frac{\pi}{h}\right]\right)$ and
        $\|\hat{v}\|_{L^2\left(\left[-\frac{\pi}{h},\frac{\pi}{h}\right]\right)}=\|v\|_{l^2_h}.$ (Parseval's equality)
        \item The sequence $u$ can be recovered from its discrete Fourier transform by the equality: \[v=\IF^{-1}(\hat{v}).\]
        \item Let $w\in l^1_h$ (defined as \eqref{eq:defl2h}). Then, the convolution product of $u$ and $v$ defined as:
        \[(v*w)_n:=h\sum_{m\in \Z} v_m w_{n-m}\]
        belongs to $l^2_h$ and
        \[\widehat{v*w}=\hat{v}\hat{w}.\]
    \end{enumerate}
    
\end{proposition}

\subsection{Standard Finite difference schemes}
\label{sec:FD}

In this section, we recall the three possible variants of two-point finite difference schemes used to approximate the differentiation operator $\partial_x$. Namely, we consider the upwind, downwind and central first-order finite difference operators on a uniform grid of width $h$ in one dimension for a bilateral sequence $(v_n)_{n\in \Z}$:
\begin{align}
\label{eq:FDupSeq}
(\ID_h^+ v)_n&=\frac{v_{n+1}-v_n}{h},\\
\label{eq:FDdownSeq}
(\ID_h^- v)_n&=\frac{v_n-v_{n-1}}{h},\\
\label{eq:FDcentralSeq}
(\ID_h v)_n&=\frac{v_{n+1}-v_{n-1}}{2h}.
\end{align}
Therefore, the three possible semidiscrete counterparts of the equation \eqref{SystGen} that correspond to the operators \eqref{eq:FDupSeq}-\eqref{eq:FDcentralSeq} are:
\begin{align}
\label{eq:problemFDup}
\partial_t U^+(t)+ A(\ID_h^+ U^+(t))&=-B U^+(t),\\
\label{eq:problemFDdown}
\partial_t U ^-(t)+ A(\ID_h^- U^-(t))&=-B U ^-(t),\\
\label{eq:problemFDcentral}
\partial_t U (t)+ A(\ID_h U(t))&=-B U (t),
\end{align}
where $U: (0,\infty) \times \Z \to \IR$ is a time-dependent bilateral sequence of vectors.

All three discrete operators \eqref{eq:FDupSeq}-\eqref{eq:FDcentralSeq} are consistent i.e. they approximate the operator $\partial_x$ as $h$ approaches zero. However, we need a structure-preserving scheme, so we study the stability of \eqref{eq:problemFDup}-\eqref{eq:problemFDcentral} to choose the appropriate one.
\medbreak

\subsection{The Fourier transform of finite difference schemes and stability results}
\label{sec:FourierStability}
Within this section, we use the discrete Fourier transform to establish stability results for the finite difference schemes \eqref{eq:problemFDup}-\eqref{eq:problemFDcentral}. This analysis will guide us in selecting the most suitable scheme for the discretization of hyperbolic systems.

The subsequent lemma, pertaining to the discrete Fourier transform of the discrete operators \eqref{eq:FDupSeq}-\eqref{eq:FDcentralSeq}, can be demonstrated through direct computation:

\begin{lemma}
Let $v\in l^2_h$. The Fourier transform of the finite difference operators \eqref{eq:FDupSeq}-\eqref{eq:FDcentralSeq} are given by:
\begin{align}
\label{eq:FDupFourier}
\widehat{(\ID_h^+ v)}(\xi)&=\frac{e^{i\xi h}-1}{h}\hat{u}(\xi),\\
\label{eq:FDdownFourier}
\widehat{(\ID_h^- v)}(\xi)&=\frac{1-e^{-i\xi h}}{h}\hat{u}(\xi),\\
\label{eq:FDcentralFourier}
\widehat{(\ID_h v)}(\xi)&=\frac{e^{i\xi h}-e^{-i\xi h}}{2h}\hat{u}(\xi).
\end{align}

As an immediate consequence, the solutions of the Cauchy problems associated with \eqref{eq:problemFDup}-\eqref{eq:problemFDcentral} are, in Fourier variables:
\begin{align}
\label{eq:problemFDupSolutionFourier}
\hat{u}^+(t,\xi)&=\exp\left[\left(-A\frac{e^{i\xi h}-1}{h}-B\right)t\right] \hat{u}_0(\xi),\\
\label{eq:problemFDdownSolutionFourier}
\hat{u}^-(t,\xi)&=\exp\left[\left(-A\frac{1-e^{-i\xi h}}{h}-B\right)t\right] \hat{u}_0(\xi),\\
\label{eq:problemFDcentralSolutionFourier}
\hat{u}(t,\xi)&=\exp\left[\left(-A\frac{e^{i\xi h}-e^{-i\xi h}}{2h}-B\right)t\right] \hat{u}_0(\xi).
\end{align}
\end{lemma}

The following proposition accounts for the stability of the discrete one-dimensional hyperbolic problem in the three analysed cases. Note that a numerical scheme is called stable if there exists a constant $C_T$ depending only on $T>0$ (so it is independent of the grid width $h$), such that, for every $t\in (0,T)$ and every $u_0\in l^2_h$,
    \[\|u(t)\|_{l^2_h}\leq C_T \|u_0\|_{l^2_h}.\]
\begin{proposition}
\label{stabilityForSchemes}
    Let $A$ be a symmetric real matrix. For initial data $u_0\in l^2_h$, the problem \eqref{eq:problemFDup} is stable if the matrix $A$ is negative, the problem \eqref{eq:problemFDdown} is stable if $A$ is positive and the problem \eqref{eq:problemFDcentral} is unconditionally stable.
    
    We recall that the real matrix $A$ is called positive if, for every vector $v\in \IR^N$, the scalar product $(Av,v)_{\IR^N}=\sum_{i,j=1}^N A_{ij}v_j v_i$ is non-negative.
\end{proposition}
\begin{proof} 
    We expand the terms $e^{i\xi h}$ using Euler's relation and use Parseval's equality in \eqref{eq:problemFDupSolutionFourier}-\eqref{eq:problemFDcentralSolutionFourier}:
\begin{align}
\label{eq:problemFDupSolutionL2}
\|u^+(t)\|_{l^2_h}&\leq \max_{\xi\in [-\pi/h,\pi/h]}\left\|\exp\left[\left(\frac{A}{h}(1-\cos(\xi h))-i\frac{A}{h}\sin(\xi h)-B\right)t\right]\right\|_{\C^n\to\C^n} \|u_0\|_{l^2_h},\\
\label{eq:problemFDdownSolutionL2}
\|u^-(t)\|_{l^2_h}&\leq \max_{\xi\in [-\pi/h,\pi/h]}\left\|\exp\left[\left(-\frac{A}{h}(1-\cos(\xi h))-i\frac{A}{h}\sin(\xi h)-B\right)t\right]\right\|_{\C^n\to\C^n} \|u_0\|_{l^2_h},\\
\label{eq:problemFDcentralSolutionL2}
\|u(t)\|_{l^2_h}&\leq \max_{\xi\in [-\pi/h,\pi/h]}\left\|\exp\left[\left(-i\frac{A}{h}\sin(\xi h)-B\right)t\right]\right\|_{\C^n\to\C^n} \|u_0\|_{l^2_h},
\end{align}
where $\|\cdot\|_{\C^N\to \C^N}$ is the matrix norm associated to the Euclidean norm on $\C^N$, which, in turn, corresponds to the scalar product:
\[(u,v)_{\C^N}=(u,\bar{v})_{\IR^N}=\sum_{i=1}^n u_i \bar{v_i}.\]
Since for every $\xi\in \left[-\frac{\pi}{h},\frac{\pi}{h}\right]$ and $h>0$,  $-\frac{\sin(\xi h)}{h}$ is a real number and $(1-\cos(\xi h))\geq 0$, the conclusion follows if we prove the two following claims:
\begin{enumerate}[label={\roman*)}]
    \item For any symmetric real matrix $A$, the matrix norm $\left\|e^{i A}\right\|_{\mathbb{C}^N\to \mathbb{C}^N}$ is equal to one.
    \item If the real matrix $A$ is symmetric negative, then the matrix norm $\left\|e^A\right\|_{\mathbb{C}^N\to \mathbb{C}^N}$ is at most one.
\end{enumerate}

For proving the first claim, we fix a vector $v_0\in \C^N$ and consider the time-dependent vector $v:[0,\infty)\to \C^N$ that satisfies the following Cauchy problem:
\begin{equation}
    \label{eq:cauchyexpiA}
    \begin{cases}
        \partial_t v(t)=iA v(t),\\
        v(0)=v_0.
    \end{cases}
\end{equation}
Taking the time derivative of the Euclidean norm $\|v(t)\|^2_{\C^N}$ and using the fact that $A$ is real and symmetric, we obtain that this norm is conserved in time. Claim i) follows since $v(1)=e^{iA} v_0$. The proof of Claim ii) follows a similar logic.
 \end{proof}
In the rest of the paper, we will focus on the central finite difference scheme, since its stability does not depend on the parameters of the system \eqref{SystGen}.
\subsection{Properties of the central finite difference operator}
We state a well-known integration by parts formula for the operator $\ID_h$ defined in \eqref{eq:FDcentralSeq}, which will be useful in our computations.
\begin{proposition}\label{Prop:IBP}
Let $u,v\in l_h^2$. The following integration by parts formula holds:
\[\langle u,\ID_h v \rangle_{l^2_h}=-\langle \ID_h u, v \rangle_{l^2_h},\]
where the $l^2_h$ scalar product associated to the norm \eqref{eq:defl2h} is given by:
\[\langle u, v \rangle_{l^2_h}=h\sum_{n\in \Z} u_n v_n.\]
An immediate consequence of the integration by parts formula is that, for every $u\in l_h^2$,
\[\langle u,\ID_h u \rangle_{l^2_h}=0.\]
\end{proposition}

For a more in-depth exploration of finite difference schemes and their properties, interested readers can consult \cite{Strikwerda2004}.

\subsection{Discrete Besov spaces}
\label{sec:discrete-besov}
In this section, we establish an analogous framework for the standard Besov spaces (associated with the continuous Fourier transform) within the discrete setting introduced in the preceding sections. The definition of these discrete Besov spaces is guided by our objective of localizing the frequencies of a bilateral sequence $(v_n)_{n\in \Z}$ in such a manner that, for each $j\in \Z$, the localization $\delta_h^j v$ is designed to satisfy a Bernstein-type estimate:
\begin{equation}\label{eq:intuitiveBernstein}
    \|\ID_h (\delta_h^j v)\|_{l^2_h} \sim 2^j \|\delta_h^j v\|_{l^2_h},
\end{equation}
where $\ID_h$ is the central finite difference operator defined in \eqref{eq:FDcentralSeq}. We will formulate the rigorous form of the Bernstein estimate in Section \ref{sec:discreteBesovProperties}. Also, the interested reader could refer to \cite[Chapter 2]{HJR} for an introduction to continuous Besov spaces and their basic properties.

\subsubsection{The construction of discrete Besov spaces}
The form of the discrete central finite difference operator $\ID_h$ in Fourier variables
\begin{equation}
\label{eq:FDCentralFourierSimple}
    \widehat{(\ID_h v)}(\xi)=i\frac{\sin(\xi h)}{h} \hat{v}(\xi),
\end{equation}
suggests the following notation:
\begin{equation}
    \label{def:Fj}
    F_h(j):=\left\{\xi \in \left[-\frac{\pi}{h},\frac{\pi}{h}\right]: \left|\frac{\sin(\xi h )}{h}\right|\in \IC_j\right\},
\end{equation}
where, for every $j\in \Z$, we denote
\begin{equation}
    \label{def:Cj}
    \IC_j := \left[\frac{3}{4} 2^j, \frac{4}{3} 2^{j+1}\right].
\end{equation}

Inspired by the dyadic decomposition used to construct the continuous Besov spaces \cite[Sections 2.2 and 2.3]{HJR}, we consider a family of functions $(\varphi_j)_{j\in \Z}$ with the following properties:
\begin{align}
    \varphi_j:\left[-\frac{\pi}{h},\frac{\pi}{h}\right]\to [0,1], \quad\forall j\in \Z,\\
    \label{eq:suppVarphij}
    {\rm supp}(\varphi_j)\subseteq F_h(j), \quad \forall j \in \Z,\\
    \label{eq:sumVarphij}
    \sum_{j\in \Z} \varphi_j(\xi)=1, \quad \forall \xi \in \left[-\frac{\pi}{h},\frac{\pi}{h}\right].
\end{align}
We note that, since the family of sets $(\IC_j)_{j\in \Z}$ is locally finite, the above sum makes sense for every $\xi\in \left[-\frac{\pi}{h},\frac{\pi}{h}\right]$. Now, we can define the $j$-th frequency localization of a sequence $(v_n)_{n\in \Z}$ and the discrete homogeneous Besov spaces.

\begin{definition}[Discrete localization operators]
    \label{def:localization}
    Let $v\in l^2_h$ and $j\in \Z$. We define the  $j$-th frequency localization of $v$ as:
    \[\delta_h^j v:=\IF^{-1}(\hat{v}\varphi_j).\]
\end{definition}

\begin{definition}[Discrete Besov Spaces -- refer to {\cite[Definition 2.15]{HJR}} for the continuous case]
    Let $s\in \IR$. The discrete Besov space $\dot{B}_{h}^s$ consists of all the sequences $v\in l^2_h$ satisfying:
    \begin{equation}
    \label{eq:BesovNorm2}
    \|v\|_{\dot{B}_{h}^s}:=\sum_{j\in \Z} 2^{js} \|\delta_h^j v\|_{l_h^2}<\infty.
    \end{equation}
\end{definition}

\subsubsection{Basic properties of discrete Besov spaces}
\label{sec:discreteBesovProperties}
First, we revisit and rigorously formulate the Bernstein estimate \eqref{eq:intuitiveBernstein}:

\begin{proposition}[Bernstein estimate for central finite difference operator]
\label{rigorousBernstein}
Let $\ID_h$ be the operator defined in \eqref{eq:FDcentralSeq}. Then, there exist two universal positive constants $C,c>0$ such that, for every $h>0$, every bilateral sequence $v\in l^2_h$ and every integer $j$,
\[c\, 2^j \|\delta_h^j v\|_{l^2_h} \leq \|\ID_h \delta_h^j v\|_{l^2_h} \leq C\, 2^j \|\delta_h^j v\|_{l^2_h},\]
where $\delta_h^j$ is the localization operator introduced in Definition \ref{def:localization}.
\end{proposition}
\begin{proof}
    Taking into account \eqref{eq:FDCentralFourierSimple} and Definition \ref{def:localization}, we obtain that:
    \[\widehat{(\ID_h \delta_h^j v)}(\xi) = i\frac{\sin (\xi h)}{h} \varphi_j(\xi) \hat{v}(\xi). \]
    From \eqref{def:Fj} and \eqref{eq:suppVarphij} we  get that, for every $\xi\in {\rm supp}(\varphi_j)$,
    \[\left|\frac{\sin(\xi h)}{h}\right|\in \IC_j.\]
    Then, the conclusion follows from Parseval's equality.
\end{proof}
\begin{definition}[Frequency-restricted discrete Besov Spaces] \label{def:freqbesov}
    Let $s\in \IR$ and $\kappa$  a small enough positive constant that will be precisely fixed in the proof of Theorem \ref{thm:NumRelax}. For ${J_\varepsilon}:= \log_2\dfrac{\kappa}{\varepsilon}$, i.e. $2^{J_\varepsilon}= \dfrac{\kappa}{\varepsilon}$, we define:
    \begin{equation}
    \label{eq:BesovNormLFHF}
    \|v\|^L_{\dot{B}_{h}^s}:=\sum_{j\leq J_\varepsilon} 2^{js} \|\delta_h^j v\|_{l_h^2} \quad \text{and} \quad   \|v\|^H_{\dot{B}_{h}^s}:=\sum_{j\geq J_\varepsilon} 2^{js} \|\delta_h^j v\|_{l_h^2}.
    \end{equation}
\end{definition}

From Proposition \ref{rigorousBernstein}, using that $2^{J_\varepsilon}=\dfrac{\kappa}{\varepsilon}$, we immediately deduce the following low-high frequencies Bernstein-type inequalities.
\begin{proposition} \label{prop:Bernstein}
Let $v\in \ell^2_h$ and $s'>0$. The following Bernstein-type inequalities hold:
\begin{align} \label{BernsteinLF}
&\|v\|_{\dot{B}^{s}_{h}}^{L}\lesssim \dfrac{\kappa^{s'}}{\varepsilon^{s'}}\|v\|_{\dot{B}^{s-s'}_{h}}^{L},
\\&\|v\|_{\dot{B}^{s}_{h}}^{H}\lesssim \dfrac{\varepsilon^{s'}}{\kappa^{s'}}\|v\|_{\dot{B}^{s+s'}_{h}}^{H}.
\label{BernsteinHF}
\end{align}
\end{proposition}

Next, we prove the embedding results of discrete Besov spaces in $\dot{h}_h^s$ and $l^\infty_h$ stated in Proposition \ref{prop:embed}. One of the important implications of this embedding is that the estimates obtained for Besov norms \eqref{eq:BesovNorm2} lead to results in well-known norms. We refer to \cite[Proposition 2.39]{HJR} for a more general embedding result in the continuous framework.

\begin{proof}[Proof of Proposition \ref{prop:embed}]
    First of all, the embedding $B^{s}_h\hookrightarrow\dot{h}^s_h$ follows immediately by the definition of the $\dot{h}^s_h$ and by \eqref{eq:suppVarphij}-\eqref{eq:sumVarphij}, using the Minkowski inequality.

    In the sequel, we focus on proving the estimate \eqref{eq:l-infty-embed}. Indeed, the property \eqref{eq:sumVarphij} implies that, for every $\xi\in \left[-\frac{\pi}{h},\frac{\pi}{h}\right]$, 
    \[v(\xi)=\sum_{j\in \Z} (\delta_h^j v)(\xi).\]
    Therefore, 
    \[\|v\|_{l^\infty_h} \leq \sum_{j\in \Z} \|\delta_h^j v\|_{l^\infty_h}.\]
    From the definition \eqref{eq:BesovNorm2} of the Besov norm, it is enough to prove that:
    \begin{equation}
        \label{eq:claim_linfty_l2}
        \|\delta_h^j v\|_{l^\infty_h} \leq C\, 2^\frac{j}{2}\|\delta_h^j v\|_{l^2_h}.
    \end{equation}
    Indeed, from \eqref{eq:suppVarphij}, it follows that $\varphi_j\cdot \chi_{F_h(j)}=\varphi_j$, where $\chi_A$ stands for the characteristic function of a set $A$. The discrete Fourier inverse formula \eqref{eq:inverseDiscreteFourier} implies that:
    \[
        \begin{aligned}
        (\delta_h^j v)_n&=\frac{1}{\sqrt{2\pi}}\int_{-\frac{\pi}{h}}^{\frac{\pi}{h}} e^{i\xi h n} \hat{v}(\xi)\varphi_j(\xi)\dxi\\
        &= \frac{1}{\sqrt{2\pi}}\int_{-\frac{\pi}{h}}^{\frac{\pi}{h}} e^{i\xi h n} \hat{v}(\xi)\varphi_j(\xi)\chi_{F_h(j)}\dxi.
        \end{aligned}
    \]
    Since $\left|e^{i\xi h n}\right|=1$, the Cauchy-Schwarz inequality and Parseval's equality further imply that:
    \[
    \begin{aligned}
        \left|(\delta_h^j v)_n\right|&\leq \frac{1}{\sqrt{2\pi}}\|\hat{v}\varphi_j\|_{L^2\left(\left[-\frac{\pi}{h},\frac{\pi}{h}\right]\right)}\|\chi_{F_h(j)}\|_{L^2\left(\left[-\frac{\pi}{h},\frac{\pi}{h}\right]\right)}\\
        &= \frac{1}{\sqrt{2\pi}} \|\delta_h^j v\|_{l^2_h} \left|F_h(j)\right|^{\frac{1}{2}},
        \end{aligned}
    \]
    where $|A|$ stands for the Lebesgue measure of the set $A$. Therefore, we are left to prove that:
    \begin{equation}
        \label{eq:claimMeasureFhj}
        \left|F_h(j)\right|\leq C \cdot 2^j.
    \end{equation}
    In order to prove this claim, we observe that an element $\xi\in \left[-\frac{\pi}{h},\frac{\pi}{h}\right]$ belongs to $F_h(j)$ if and only if
    \begin{equation}
    \label{eq:observation-sinxih}
    \frac{\sin(\xi h)}{\xi h} \xi \in \left[\frac{3}{4}2^j,\frac{4}{3}2^{j+1}\right].
    \end{equation}
    Next, we fix a constant $c\in \left(0,\frac{\pi}{2}\right)$ and notice from the plot in Figure \ref{fig:plot-sinx-x} that there exists another constant $M_c\in (0,1)$ such that, for every $x\in [-\pi+c,\pi+c]$,
    \begin{equation}
    \label{eq:estimate-sinx-x}
    M_c\leq \frac{\sin(x)}{x}\leq 1.
    \end{equation}
    Therefore, if $\xi h\in [-\pi+c,\pi-c]$, then \eqref{eq:observation-sinxih} implies that $\xi\in \left[\frac{3}{4}2^j,\frac{4}{3 M_c}2^{j+1}\right]$. We can now estimate the Lebesgue measure of a part of $F_h(j)$:
    \begin{equation}
    \label{eq:FhjMeasureEst1}
    \left|F_h(j)\cap \left[\frac{-\pi+c}{h},\frac{\pi-c}{h}\right]\right|\leq 2^j\left(\frac{8}{3M_c}-\frac{3}{4}\right).
    \end{equation}
    Thence, we consider the case $\xi h\in [\pi-c, \pi]$, which means that $\pi-\xi h\in [0,c]\subset [-\pi+c,\pi-c]$. In this case, we have
    \[\frac{\sin(\pi-\xi h)}{\pi-\xi h} \in [M_c,1].\]
    Therefore, if $\xi$ is such that \eqref{eq:observation-sinxih} holds, then the equality $\sin(x)=\sin(\pi-x)$ implies that:
    \[\frac{\pi}{h}-\xi = \frac{\sin(\xi h)}{\xi h}\xi \frac{\pi-\xi h}{\sin(\pi-\xi h)} \in \left[\frac{3}{4}2^j,\frac{4}{3M_c}2^{j+1}\right]. \]
    This leads us to an estimate of the Lebesgue measure of a second part of $F_h(j)$:
    \begin{equation}
    \label{eq:FhjMeasureEst2}
    \left|F_h(j)\cap \left[\frac{\pi-c}{h},\frac{\pi}{h}\right]\right|\leq 2^j\left(\frac{8}{3M_c}-\frac{3}{4}\right).
    \end{equation}
    By matters of symmetry, we arrive also to an estimate regarding the third part of $F_h(j)$:
    \begin{equation}
    \label{eq:FhjMeasureEst3}
    \left|F_h(j)\cap \left[\frac{-\pi}{h},\frac{-\pi+c}{h}\right]\right|\leq 2^j\left(\frac{8}{3M_c}-\frac{3}{4}\right).
    \end{equation}
    Combining \eqref{eq:FhjMeasureEst1}, \eqref{eq:FhjMeasureEst2} and \eqref{eq:FhjMeasureEst3} we obtain the claim \eqref{eq:claimMeasureFhj}, which finishes the proof.
\end{proof}

\section{Proof Theorem \ref{thm:NumHyp}: Hypocoercivity for semi-discretized hyperbolic systems}
\label{sec:proof-asymptotic}
This section is dedicated to the proof of the large-time asymptotic result \eqref{e:decayThm12}. Across the paper, the notations $E\sim F$ and $E\lesssim F$ signify that there exists a universal constant $C>1$ such that $\frac{1}{C} F \leq E\leq C F$ and $E\leq C F$, respectively.
\subsection{Decay for $\ID_hU$}
To derive sharp time-decay estimates for the semi-discrete hyperbolic system \eqref{mainr:Systn}, inspired by \cite{CBSZ}, we consider the Lyapunov functional
\begin{equation}\label{def:GeneralLyaNoF}
\begin{aligned}
  \cL(t):= \|U (t)\|_{h_h^1}^2+\eta_{0} t\|\ID_h U (t)\|_{\ell_h^2}^2+\cI(t),
  \end{aligned}
  \end{equation}
where the corrector term  $\cI(t)$ is defined by
\begin{align}
&\cI(t):= \sum_{k=1}^{N-1} \varepsilon_k \bigl(BA^{k-1}U ,BA^k\ID_hU \bigr)_{\ell_h^2},\label{I}
\end{align}
with positive constants $\eta_{0}$ and $\var_{i}$, $i=1,2,...,k-1$, to be determined later. Then, we compute the time-derivative of $\mathcal{L}$ and show that $\dfrac{d}{dt}\mathcal{L}\leq 0$, which, in turn, leads to the desired decay estimates.
Indeed, using that $\widetilde{B}$ is strongly dissipative (property \eqref{BD}) and Proposition \ref{Prop:IBP}, taking the scalar product of \eqref{mainr:Systn} with $U$, we get
\begin{equation}\label{decay1}
    \dfrac{d}{dt}\|U (t)\|_{\ell_h^2}^2+2\lambda\|U _{2}(t)\|_{\ell_h^2}^2\leq 0.
\end{equation}
Similarly, applying the linear operator $\ID_h$ to \eqref{mainr:Systn} and taking the scalar product with $\ID_hU$, we obtain
\begin{equation}\label{decay2}
    \dfrac{d}{dt}\|\ID_hU (t)\|_{\ell_h^2}^2+2\lambda\|\ID_hU _{2}(t)\|_{\ell_h^2}^2\leq 0.
\end{equation}
Moreover, this implies the time-weighted estimate
\begin{equation}\label{decay3}
    \dfrac{d}{dt}\left(\eta_{0} t\|\ID_hU (t)\|_{l_h^2}^2\right)+2\lambda\eta_0t \|\ID_hU _{2}(t)\|_{l_h^2}^2\leq \eta_{0}\|\ID_hU (t)\|_{l_h^2}^2.
\end{equation}
Gathering \eqref{decay1}, \eqref{decay2} and \eqref{decay3}, we obtain
\begin{equation}
    \dfrac{d}{dt}\Big(\|U (t)\|_{h_h^1}^2+\eta_{0} t\|\ID_h U (t)\|_{l_h^2}^2\Big)+2\lambda \|U _{2}(t)\|_{l_h^2}^2+2\lambda(1+ \eta_{0} t) \|\ID_hU _{2}(t)\|_{l_h^2}^2\leq \eta_{0}\|\ID_hU (t)\|_{l_h^2}^2.\label{L}
\end{equation}
As one can observe in \eqref{L}, there are no dissipative effects for the component $U_1$. To recover such dissipation, we take the time-derivative of the corrector term $\mathcal{I}$.
\begin{lemma}[Time-derivative of $\cI$]\label{lem:corrector}
For any positive constant $\var_{0}$, there exists a sequence $\{\var_{k}\}_{k=1,...,N-1}$ of small positive constants such that
\begin{align}
\dfrac{d}{dt}\cI(t) +\frac{1}{2}\sum_{k=1}^{N-1} \varepsilon_k\|BA^k\ID_hU (t)\|_{l_h^2}^2 \leq \var_{0}\|U _{2}(t)\|_{l_h^2}^2+\var_{0}\| \ID_hU _{2}(t)\|_{l_h^2}^2.\label{ddtI}
\end{align}
\end{lemma}
The proof of Lemma \ref{lem:corrector} is a direct adaptation of the computations done in \cite{BZ} to our discrete setting. Its proof is relegated to the appendix, Section \ref{A:lem:correct}.

Next, we fix suitably small $\varepsilon_{k}, k=1,2,...,N-1$, such that \eqref{ddtI} holds and,  by applying the Cauchy-Schwarz inequality in \eqref{I}, one has:
 \begin{align}
 & \mathcal{L}(t)\sim \|U(t)\|_{h_h^1}^2+\eta_{0} t\|\ID_h U(t)\|_{\ell_h^2}^2. \label{Lsim}
 \end{align}
Combining the energy inequality \eqref{L} and the estimate \eqref{ddtI} of  the corrector term, we obtain
\begin{equation}\label{mmm}
\begin{aligned}
    &\dfrac{d}{dt}\mathcal{L}(t)+\lambda \|U _{2}(t)\|_{\ell_h^2}^2+\lambda(1+ 2 \eta_{0} t) \|\ID_hU _{2}(t)\|_{\ell_h^2}^2+\frac{1}{2}\sum_{k=1}^{N-1} \varepsilon_k\|BA^k\ID_h{U}(t)\|_{\ell_h^2}^2\\
    &\leq \eta_{0}\|\ID_h U (t)\|_{\ell_h^2}^2+\var_{0}\|U_{2}(t)\|_{\ell_h^2}^2+\var_{0}\|\ID_hU _{2}(t)\|_{\ell_h^2}^2.
\end{aligned}
\end{equation}
From \cite[Lemma 1]{BZ}, we have that, for $y\in \mathbb{C}^N$, the function
\begin{align}\label{intro:normeq}
\mathcal{N}(y):=\bigg(\sum_{k=0}^{N-1} |BA^{k}y|^2\bigg)^{\frac{1}{2}} \quad \text{defines a norm on} \:\mathbb{C}^N,\end{align}
which, by standard properties of finite-dimensional spaces, is equivalent to any other norm, in particular to the Euclidean one. Using this norm equivalence, in the case $y=D_h U$, we get that:
$$
\lambda\|\ID_h U _{2}(t)\|_{\ell^2_h}^2+\sum_{k=1}^{N-1} \varepsilon_k\|BA^k\ID_hU (t)\|_{\ell^2_h}^2\geq \frac{\var_{*}}{C_{2}}\|\ID_h U (t)\|_{\ell^2_h}^2,
$$
with $\var_{*}:=\min\{\lambda, \var_{1}, \var_{1},...,\var_{N-1}\}$ and $C_{2}>0$ a constant depending only on $(A,B)$ and $N$.  Therefore, to ensure the coercivity of \eqref{mmm}, we adjust the coefficients appropriately as
$$
0<\eta_{0}<\frac{\var_{*}}{4C_{2}} ,\quad\quad 0<\var_{0}<\frac{\lambda}{2},
$$
such that 
\begin{align}
\dfrac{d}{dt}\mathcal{L}(t) +\frac{\lambda}{2} \| U _{2}(t)\|_{\ell^2_h}^2+\lambda \left(\frac{1}{2}+\eta_{0} t\right) \| \ID_h U _{2}(t)\|_{\ell^2_h}^2+\frac{\var_{*}}{4C_{2}}\|\ID_h U (t)\|_{\ell^2_h}^2\leq 0.\label{Lineq}
\end{align} 
Therefore, by \eqref{Lineq}, we have  $\mathcal{L}(t)\leq \mathcal{L}(0)$, which by the equivalence in \eqref{Lsim}  leads to
\begin{equation}
\begin{aligned}
&\|U (t)\|_{\ell^2_h}+(1+t)^{\frac12}\|\ID_h U (t)\|_{\ell^2_h}\leq C\|U_{0}\|_{h^1}.\label{uptoU1}
\end{aligned}
\end{equation}

\subsection{Time-decay estimates for $U_2$}

Taking the inner product of the equation satisfied by $U_2$ in \eqref{mainr:Systn} with $U _{2}$ and using the property \eqref{BD}, we get
\begin{equation}\label{ddtU22}
    \begin{aligned}
    &\frac{d}{dt}\|U _{2}(t)\|_{\ell^2_h}^2+2\lambda\|U _{2}(t)\|_{\ell^2_h}^2\lesssim \|\ID_hU (t)\|_{\ell^2_h}\|U _{2}(t)\|_{\ell^2_h}.
    \end{aligned}
\end{equation}
Dividing the above inequality \eqref{ddtU22} by $\sqrt{\|U _{2}(t)\|_{\ell^2_h}^2+\epsilon}$, employing Gr\"onwall's inequality and then letting $\epsilon\rightarrow0$, we obtain
\begin{align}
&\|U _{2}(t)\|_{\ell^2_h}\lesssim e^{-\lambda t}\|U_{0,2}\|_{\ell^2_h}+\int_{0}^{t}e^{-\lambda(t-\tau)} \|\ID_hU(\tau)\|_{\ell^2_h} ~d\tau.\label{this111}
\end{align}
This inequality, together with Lemma \ref{A:lem:estimateExponential}, leads to
\begin{align}
\|U _{2}(t)\|_{\ell^2_h}&\leq e^{-\lambda t}\|U_{0,2}\|_{\ell^2_h}+\|U_0\|_{h^1_h}\int_{0}^{t}e^{-\lambda (t-\tau)}(1+\tau)^{-\frac{1}{2}}\,d\tau \lesssim (1+t)^{-\frac{1}{2}}\|U_{0}\|_{h^1_h}.\nonumber
\end{align}
The proof of Theorem \ref{thm:NumHyp} is concluded. \qed

\section{Proof Theorem \ref{thm:NumRelax}: Numerical relaxation limit}
\label{sec:proof-relaxation-result}
In this section, we provide the proof of the relaxation estimate \eqref{StrongRelax}.
\subsection{Uniform-in-$\varepsilon$ estimates for \eqref{mainR:inEuler}}
\label{sec:unif-in-epsilon-estimates}
Applying the localisation operator $\delta^j_h$ to the system \eqref{mainR:inEuler}, we obtain
\begin{equation}\label{syst:linEulerj} 
\left\{\begin{aligned}
   & \d_t \rho^\varepsilon_j + \ID_hu^\varepsilon_j=0,
    \\ & \varepsilon^2\d_t u^\varepsilon_j + \ID_h\rho^\varepsilon_j +u^\varepsilon_j=0,
\end{aligned}\right.
\end{equation}
where we used the notation $f_j:=\delta^j_hf$ for any $f\in \ell^2_h$. From here, the analysis is inspired by the computations done in \cite{CBD3}, but with certain modifications aimed to sharpen, in this linear setting, the convergence ratio to $\mathcal{O}(\eps^2)$, instead of $\mathcal{O}(\eps)$.
 \medbreak
\noindent\textbf{Low-frequency analysis: $j\leq J_\varepsilon$.}
\medbreak
Defining the damped mode $w^\varepsilon= \ID_h\rho^\varepsilon +u^\varepsilon$ and inserting it in \eqref{syst:linEulerj}, we have
\begin{equation}\label{syst:linEulerj-damped-mode}
\left\{\begin{aligned}
   & \d_t \rho^\varepsilon_j - \ID_h^2\rho^\varepsilon_j =-\ID_hw^\varepsilon_j,
    \\ & \d_t w^\varepsilon_j + \frac{1}{\varepsilon^2}w^\varepsilon_j=\ID_h^3\rho^\varepsilon_j-\ID_h^2w^\varepsilon_j,
\end{aligned}\right.
\end{equation}
Taking the scalar product of the first equation of \eqref{syst:linEulerj-damped-mode} with $\rho_j^\varepsilon$, we obtain, by the Cauchy-Schwarz inequality, that
\begin{align}
    \dfrac{1}{2}\dfrac{d}{dt}\|\rho_j^\varepsilon\|_{\ell^2_h}^2+\|\ID_h\rho_j^\varepsilon\|_{\ell^2_h}^2 &\leq \|\ID_hw_j^\varepsilon\|_{\ell^2_h}\|\rho_j^\varepsilon\|_{\ell^2_h}.
\end{align}
Using the Bernstein estimate in Proposition \ref{rigorousBernstein}, we have
\begin{align}
    \dfrac{1}{2}\dfrac{d}{dt}\|\rho_j^\varepsilon\|_{\ell^2_h}^2+2^{2j}\| \rho_j^\varepsilon\|_{\ell^2_h}^2 &\lesssim \|\ID_hw_j^\varepsilon\|_{\ell^2_h}\|\rho_j^\varepsilon\|_{\ell^2_h}.
\end{align}
We can now apply Lemma \ref{SimpliCarre} which yields 
\begin{align}\label{eq:12}
\|\rho_j^\varepsilon(T)\|_{\ell^2_h}+2^{2j}\int_0^T\| \rho_j^\varepsilon\|_{\ell^2_h} \lesssim \|\rho_{0,j}^*\|_{\ell^2_h}+\int_0^t\|\ID_hw_j^\varepsilon\|_{\ell^2_h}.
\end{align}
Then, for $s\in \R$, multiplying \eqref{eq:12} by $2^{js}$ and summing on $j\leq J_\varepsilon$, we obtain, with the notations in Definition \ref{def:freqbesov}, that
\begin{align} \label{LFenergrho}
\|\rho^\varepsilon(T)\|_{B^{s}_{h}}^L+\|\rho^\varepsilon\|_{L^1_T(B^{s+2}_{h})}^L &\lesssim \|\rho_{0}^*\|^L_{B^{s}_{h}}+\|w^\varepsilon\|^L_{L^1_T(B^{s+1}_{h})}.
\end{align}
Performing similar estimates for $w^\varepsilon_j$, we obtain
\begin{align} \label{LFenergw}
\|w^\varepsilon(T)\|_{B^{s-1}_{h}}^L+\dfrac{1}{\varepsilon^2}\|w^\varepsilon\|_{L^1_T(B^{s-1}_{h})}^L &\lesssim  \| w_{0}^*\|^L_{B^{s-1}_{h}}+\|\rho^\varepsilon\|^L_{L^1_T(B^{s+2}_{h})}+\|w^\varepsilon\|^L_{L^1_T(B^{s+1}_{h})}.
\end{align}

Using the low-frequency Bernstein inequality \eqref{BernsteinLF}, we have
\begin{align} 
\|w^\varepsilon\|^L_{L^1_T(B^{s+1}_{h})}\label{BernLF2}
 &\lesssim \dfrac{\kappa^2}{\varepsilon^2}\|w^\varepsilon\|^L_{L^1_T(B^{s-1}_{h})}.
\end{align}
Summing \eqref{LFenergrho} and \eqref{LFenergw}, using \eqref{BernLF2} and choosing $\kappa$ suitably small, we obtain
\begin{equation} \label{LFenergCCL}
\begin{aligned}
\|\rho^\varepsilon(T)\|^L_{B^{s}_{h}}+\| w^\varepsilon(T)\|^L_{B^{s-1}_{h}}+\dfrac{1}{\varepsilon^2}\|w^\varepsilon\|^L_{L^1_T(B^{s-1}_{h})} &\lesssim \|\rho_{0}^*\|^L_{B^{s}_{h}}+\|w_{0}^*\|^L_{B^{s-1}_{h}}
    \\&\lesssim \|\rho_{0}^*\|^L_{B^{s}_{h}}+\|u_{0}^*\|^L_{B^{s-1}_{h}}.
\end{aligned}
\end{equation}
\medbreak
\noindent\textbf{High-frequency analysis: $j> J_\varepsilon$.}
\medbreak
We define the following Lyapunov functional
\begin{align}
    \label{Lyahf}
    \mathcal{L}^\varepsilon_j= \|(\rho_j^\varepsilon,\varepsilon u_j^\varepsilon)\|_{\ell^2_h}^2+c_1 2^{-2j} \langle u_j^\varepsilon, \ID_h\rho^\varepsilon_j\rangle_{\ell_h^2},
\end{align}
where $c_1$ is a small constant which will be chosen later. By Bernstein’s inequality (Proposition \ref{rigorousBernstein}) and using that $j\geq J_\varepsilon$, we obtain
\begin{align*} 
   2^{-2j} \langle u_j^\varepsilon, \ID_h\rho^\varepsilon_j\rangle_{\ell_h^2}  &\lesssim 2^{-2j}(\|u_j^\varepsilon\|_{\ell_h^2}^2+2^{2j}\| \rho^\varepsilon_j \|_{\ell_h^2}^2)
    \\&= 2^{-2j}\|u_j^\varepsilon\|_{\ell_h^2}^2+\|\rho_j^\varepsilon\|_{\ell_h^2}^2 
    \\&\lesssim \dfrac{\varepsilon^2}{\kappa^2}\|u_j^\varepsilon\|_{\ell_h^2}^2+\|\rho_j^\varepsilon\|_{\ell_h^2}^2,
\end{align*}
and thus $\mathcal{L}_j^\varepsilon\sim \|(\rho_j^\varepsilon,\varepsilon u_j^\varepsilon)\|_{\ell^2_h}^2$ for a suitably small constant $c_1$ depending on $\kappa$.

We now compute the time derivative of $\mathcal{L}^\varepsilon_j$. Concerning the first term, we have from \eqref{syst:linEulerj}:
\begin{align}\label{eqhf3}
    \dfrac{1}{2}\dfrac{d}{dt}\|(\rho^\varepsilon_j,\varepsilon u_j^\varepsilon)\|_{\ell^2_h}^2+\| u^\varepsilon_j\|_{\ell^2_h}^2=0.
\end{align}
For the second term, we get
\begin{align}
  2^{-2j} \dfrac{d}{dt}\langle u_j^\varepsilon, \ID_h\rho^\varepsilon_j\rangle_{\ell_h^2} +  2^{-2j} \dfrac{1}{\varepsilon^2}\|\ID_h\rho^\varepsilon_j\|_{\ell^2_h}^2 = 2^{-2j}\|\ID_hu^\varepsilon_j\|_{\ell^2_h}^2- 2^{-2j}\dfrac{1}{\varepsilon^2} \langle u_j^\varepsilon, \ID_h\rho^\varepsilon_j\rangle_{\ell_h^2}.
\end{align}
Using the Cauchy-Schwarz and Young inequalities and taking into account that $2^{-2j}\leq \varepsilon^2/\kappa^2$, we obtain
\begin{equation} \label{eqhf4}
\begin{aligned} 
  2^{-2j} \dfrac{d}{dt} \langle u_j^\varepsilon, \ID_h\rho^\varepsilon_j\rangle_{\ell_h^2} +  \dfrac{1}{\varepsilon^2}\|\rho^\varepsilon_j\|_{\ell^2_h}^2 &\lesssim \|u^\varepsilon_j\|_{\ell^2_h}^2+ \dfrac{2^{-2j}}{2\varepsilon^2c_2}\|u^\varepsilon_j\|_{\ell^2_h}^2+\dfrac{2^{-2j}c_2}{2\varepsilon^2}\|D_h\rho^\varepsilon_j\|_{\ell^2_h}^2
  \\&\lesssim \|u^\varepsilon_j\|_{\ell^2_h}^2+ \dfrac{1}{2c_2\kappa^2}\|u^\varepsilon_j\|_{\ell^2_h}^2+\dfrac{c_2}{2\varepsilon^2}\|\rho^\varepsilon_j\|_{\ell^2_h}^2,
\end{aligned}
\end{equation}
where $c_2>0$ is small enough in order for $\frac{1}{\eps^2}\|\rho_j^\eps\|_{\ell_h^2}^2$ to absorb the term $\frac{c_2}{2\eps^2}\|\rho_j^\eps\|_{\ell_h^2}^2$.

Since $\varepsilon<1$ and $\kappa$ is fixed, multiplying \eqref{eqhf4} by a constant $c_1$ small enough and adding it to 
\eqref{eqhf3}, we obtain
\begin{align}
    \dfrac{1}{2}\dfrac{d}{dt}\mathcal{L}^\varepsilon_j \lesssim -\dfrac{1}{\varepsilon^2}\|(\rho^\varepsilon_j,\varepsilon u_j^\varepsilon)\|_{\ell^2_h}^2.
\end{align}
Then, using that $\mathcal{L}_j^\varepsilon\sim \|(\rho_j^\varepsilon,\varepsilon u_j^\varepsilon)\|_{\ell^2_h}^2$ and Lemma \ref{SimpliCarre},
we get
\begin{align} \label{eqhf5}
    \|(\rho^\varepsilon_j,\varepsilon u_j^\varepsilon)(T)\|_{\ell^2_h} + \dfrac{1}{\varepsilon^2}\|(\rho^\varepsilon_j,\varepsilon u_j^\varepsilon)\|_{L^1_T(\ell^2_h)} \lesssim \|(\rho^\varepsilon_{0,j},\varepsilon u_{0,j}^*)\|_{\ell^2_h}.
\end{align}
Multiplying \eqref{eqhf5} by $2^{js}$ and summing the resulting equation for $j\geq J_\varepsilon$, we obtain
\begin{align} \label{HfenergCCL}
    \|(\rho^\varepsilon,\varepsilon u^\varepsilon)(T)\|^H_{B^{s}_{h}}+\dfrac{1}{\varepsilon^2}\|(\rho^\varepsilon,\varepsilon u^\varepsilon)\|^H_{L^1_T(B^{s}_{h})} &\lesssim \|(\rho_{0}^*,\varepsilon u_{0}^*)\|^H_{B^{s}_{h}}.
\end{align}
Recalling that $w^\varepsilon=\ID_h\rho^\varepsilon+u^\varepsilon$, thanks to Propositions \ref{rigorousBernstein} and \ref{prop:Bernstein} it is easy to see that
\begin{equation} \label{boundhfw}
\begin{aligned}\|w^\varepsilon\|^H_{L^1_T(B^{s-1}_{h})}&\lesssim \|\rho^\varepsilon\|^H_{L^1_T(B^{s}_{h})}+\|u^\varepsilon\|^H_{L^1_T(B^{s-1}_{h})} 
    \\ &\lesssim \|\rho^\varepsilon\|^H_{L^1_T(B^{s}_{h})}+\dfrac{\varepsilon}{\kappa}\|u^\varepsilon\|^H_{L^1_T(B^{s}_{h})} 
    \\& \lesssim \varepsilon^2\|(\rho_{0}^*,\varepsilon u_{0}^*)\|^H_{B^{s}_{h}}.
\end{aligned}
\end{equation}

\subsection{Error estimates analysis}
We can now justify the relaxation estimate \eqref{StrongRelax}.

\medbreak
We define the error unknown $\bar \rho := \rho^\varepsilon-\rho$, it satisfies
\begin{align}
    \label{eq:error}
    \d_t \bar \rho -\ID_h^2\bar{\rho}=-\ID_hw^\varepsilon.
\end{align}
As in the proof of \eqref{LFenergrho}, we are led to
\begin{align} \label{ErrorEnerg}
    \|\bar \rho(T)\|_{B^{s-2}_{h}}+\|\bar\rho\|_{L^1_T(B^{s}_{h})} &\lesssim \|\rho_0^*-\rho_0\|_{B^{s-2}_{h}} +\|w^\varepsilon\|_{L^1_T(B^{s-1}_{h})}.
\end{align}
Then, using \eqref{LFenergCCL} and \eqref{boundhfw}, we obtain
\begin{equation}\label{BernHF12}
\begin{aligned}\|w^\varepsilon\|_{L^1_T(B^{s-1}_{h})}&= \|w^\varepsilon\|^L_{L^1_T(B^{s-1}_{h})}+\|w^\varepsilon\|^H_{L^1_T(B^{s-1}_{h})} 
    \\&\lesssim \varepsilon^2\left(\|\rho_{0}^*\|^L_{B^{s}_{h}}+\|u_{0}^*\|^L_{B^{s-1}_{h}}\right)+ \varepsilon^2\|(\rho_{0}^*,\varepsilon u_{0}^*)\|^H_{B^{s}_{h}}.
\end{aligned}
\end{equation}
Using the definition of the $(s',h)$-truncation and Theorem \ref{thm:unif_Besov_estimate} we have \begin{align}
     \label{illprep}
     \|\rho_0^*-\rho_0\|_{\dot{B}^{s-2}_{h}}\lesssim \varepsilon^2.
 \end{align}
Inserting \eqref{BernHF12} and \eqref{illprep} into \eqref{ErrorEnerg} and using Theorem \ref{thm:unif_Besov_estimate} again concludes the proof of Theorem \ref{thm:NumRelax}.

\section{Proof of Theorem \ref{thm:unif_Besov_estimate}: Uniform Besov estimates with respect to the grid width}
\label{sec:proof_thm_unif_Besov_estimate}
In this section, we provide the proof of Theorem \ref{thm:unif_Besov_estimate} concerning uniform Besov estimates with respect to the grid width $h$, for regular enough functions.

\begin{proof}[Proof of Theorem \ref{thm:unif_Besov_estimate}]
    We recall that the discrete $\dot{B}_{h}^s$-norm of $\mathcal{T}_h u$, by definition \eqref{eq:BesovNorm2}, reads
    \begin{equation}
       \label{eq:BesovTu_1}
       \|\mathcal{T}_h u\|_{\dot{B}_{h}^s} = \sum_{j\in \Z} 2^{js} \|\delta_h^j \mathcal{T}_h u\|_{\ell_h^2}.
    \end{equation}
    Taking into account that, by definition \eqref{def:truncation}, the function $u$ and the bilateral sequence $\mathcal{T}_h u$ have essentially the same Fourier transform, we use Parseval's equality to write
    \begin{equation}
    \label{eq:deltaTu_1}
    \begin{aligned}
        \|\delta_h^j \mathcal{T}_h u\|_{\ell_h^2}^2 &= \int_{-\frac{\pi}{h}}^{\frac{\pi}{h}}  (\hat{u}(\xi))^2 (\varphi_j(\xi))^2 \dxi \\
        & = \int_{-\frac{\pi}{h}}^{\frac{\pi}{h}}  (\hat{u}(\xi))^2 (1+|\xi|^{2s'}) (\varphi_j(\xi))^2 \frac{1}{1+|\xi|^{2s'}} \dxi.
        \end{aligned}
    \end{equation}
    Now, since ${\rm supp}(\varphi_j) \subseteq F_h(j)$, it means by \eqref{def:Fj} that, if $\varphi_j(\xi) \neq 0$, then 
    \[\left|\frac{\sin(\xi h)}{\xi h}\right| |\xi| \geq \frac{3}{4} 2^j.\]
    Since, $\left|\frac{\sin(x)}{x}\right|\leq 1$, for all $x\in [-\pi,\pi]$, we obtain
    \[\varphi_j(\xi) \neq 0 \Rightarrow |\xi|\geq \frac{3}{4} 2^j.\]
    This fact, together with \eqref{eq:deltaTu_1} and $\varphi_j(\xi)\in [0,1]$, for all $\xi\in[-\pi/h,\pi,h]$, leads to
    \begin{equation}
        \label{eq:deltaTu_2}
         \|\delta_h^j \mathcal{T}_h u\|_{\ell_h^2}^2 \leq \frac{1}{1+\left(\frac{3}{4}\right)^{2s'} 2^{2js'}} \int_{-\frac{\pi}{h}}^{\frac{\pi}{h}}  (\hat{u}(\xi))^2 (1+|\xi|^{2s'}) \dxi.
    \end{equation}
    Applying Parseval's equality again, we deduce that
    \begin{equation}
        \label{eq:deltaTu_3}
         \|\delta_h^j \mathcal{T}_h u\|_{\ell_h^2} \leq \frac{C_{s'}}{1+2^{js'}} \|u\|_{H^{s'}(\IR)}.
    \end{equation}
    Inserting this inequality into \eqref{eq:BesovTu_1}, we get:
    \begin{equation}
        \label{eq:BesovTu_2}
         \|\mathcal{T}_h u\|_{\dot{B}^s_{h}} \leq C_{s'} \|u\|_{H^{s'}(\IR)} \sum_{j\in \Z} \frac{2^{js}}{1+2^{js'}}.
    \end{equation}
    We now claim that the hypotheses of Theorem \ref{thm:unif_Besov_estimate} imply that the series above is convergent. Indeed, one has
    \[\sum_{j\in \Z} \frac{2^{js}}{1+2^{js'}} = \sum_{j \leq 0} \frac{2^{js}}{1+2^{js'}} + \sum_{j > 0} \frac{2^{js}}{1+2^{js'}} \lesssim \sum_{j \leq 0} 2^{js} + \sum_{j > 0} 2^{j(s-s')}, \]
    which converges provided that $s\in (0,s')$.    
\end{proof}

\section{Numerical simulations}
In this section, we showcase a set of numerical experiments validating our theoretical findings. The simulations in Section \ref{sec:hypo-experiments}, carried out using the \verb|NumPy| and \verb|Matplotlib| Python libraries \cite{harris2020array,Hunter:2007}, confirm the sharpness of the polynomial decay verified by the solutions of the system \eqref{mainr:Systn} (as per Theorem \ref{thm:NumHyp}). Furthermore, the experiments carried out in Section \ref{sec:relax-experiments}  show that the order of convergence $\mathcal{O}(\eps^2)$ obtained in Theorem \eqref{thm:NumRelax} (more specifically, in Corollary \ref{cor:lInftyRelax}) is, in turn, sharp.

\subsection{The numerical hypocoercivity property}
\label{sec:hypo-experiments}
The plot depicted in Figure \ref{fig:decay} validates the polynomial large-time decay estimate \eqref{e:decayThm1}, for a particular instance of \eqref{mainr:Systn} -- namely the linearization of the compressible Euler system \eqref{mainR:inEuler} -- exhibiting a decay rate of exactly $(1+t)^{-\frac{1}{2}}$. The initial data that we used in the simulation is obtained by a cut-off near infinity of the function:
\begin{equation}
\label{def:particularInitialData}
\widetilde{\rho_0}(x)=\widetilde{u_0}(x)= \frac{1}{\sqrt[4]{x^2+10^{-6}}}.
\end{equation}

\begin{figure}[h!]
    \centering
    \includegraphics[scale=0.2]{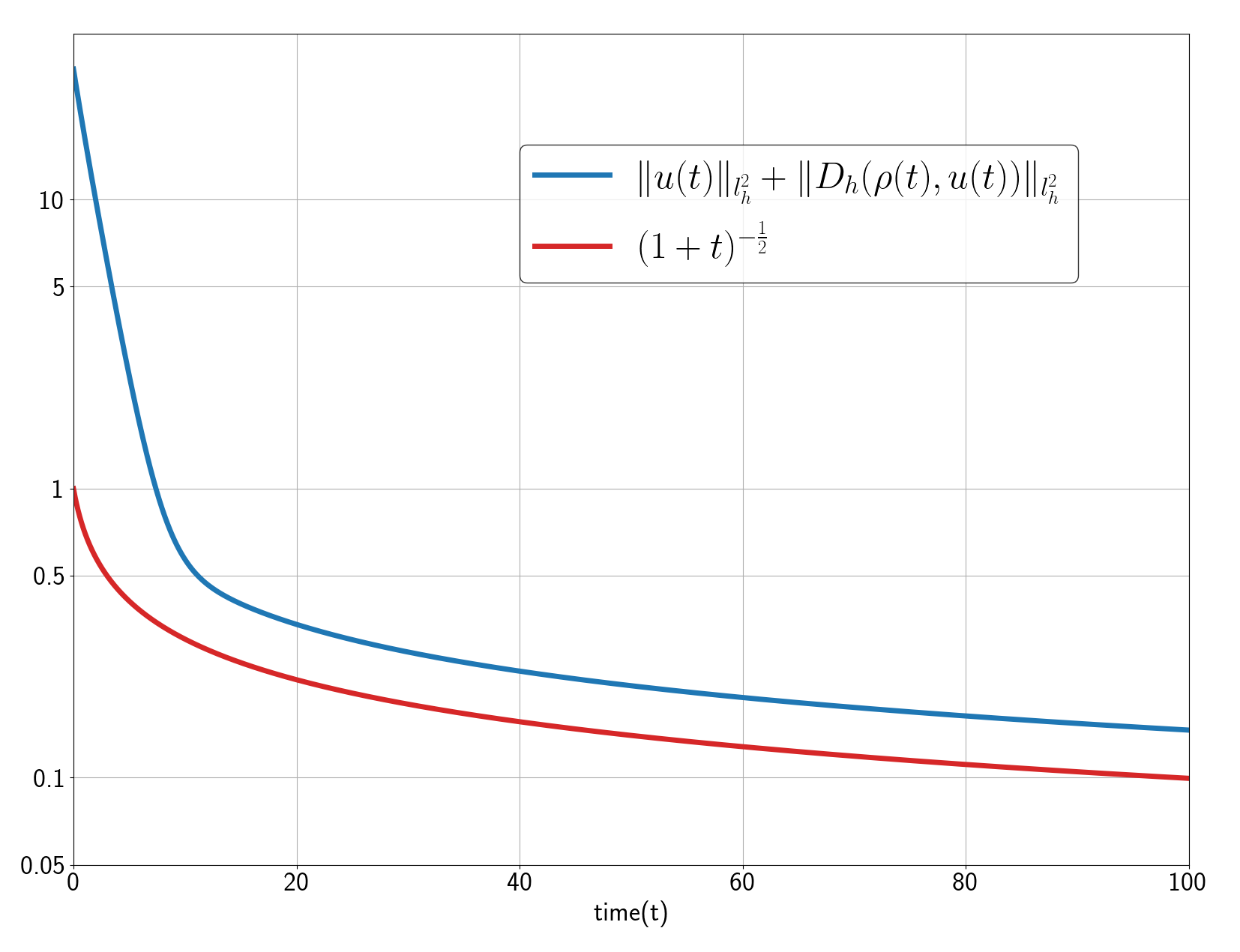}
    \caption{The semi-log plot of the large time behaviour of the solution of \eqref{mainR:inEuler} with parameters $\varepsilon=1$ and $h=2^{-4}$.}
    \label{fig:decay}
\end{figure}

\subsection{The relaxation limit -- error estimates}
\label{sec:relax-experiments}
The objective of the overlapped plot in Figure \ref{fig:overlapped} is to demonstrate that the solutions of \eqref{mainR:inEuler} effectively approximate the discrete heat equation \eqref{mainR:heat}$_1$ for small $\varepsilon$.

\begin{figure}[h!]
    \centering
    \includegraphics[scale=0.24]{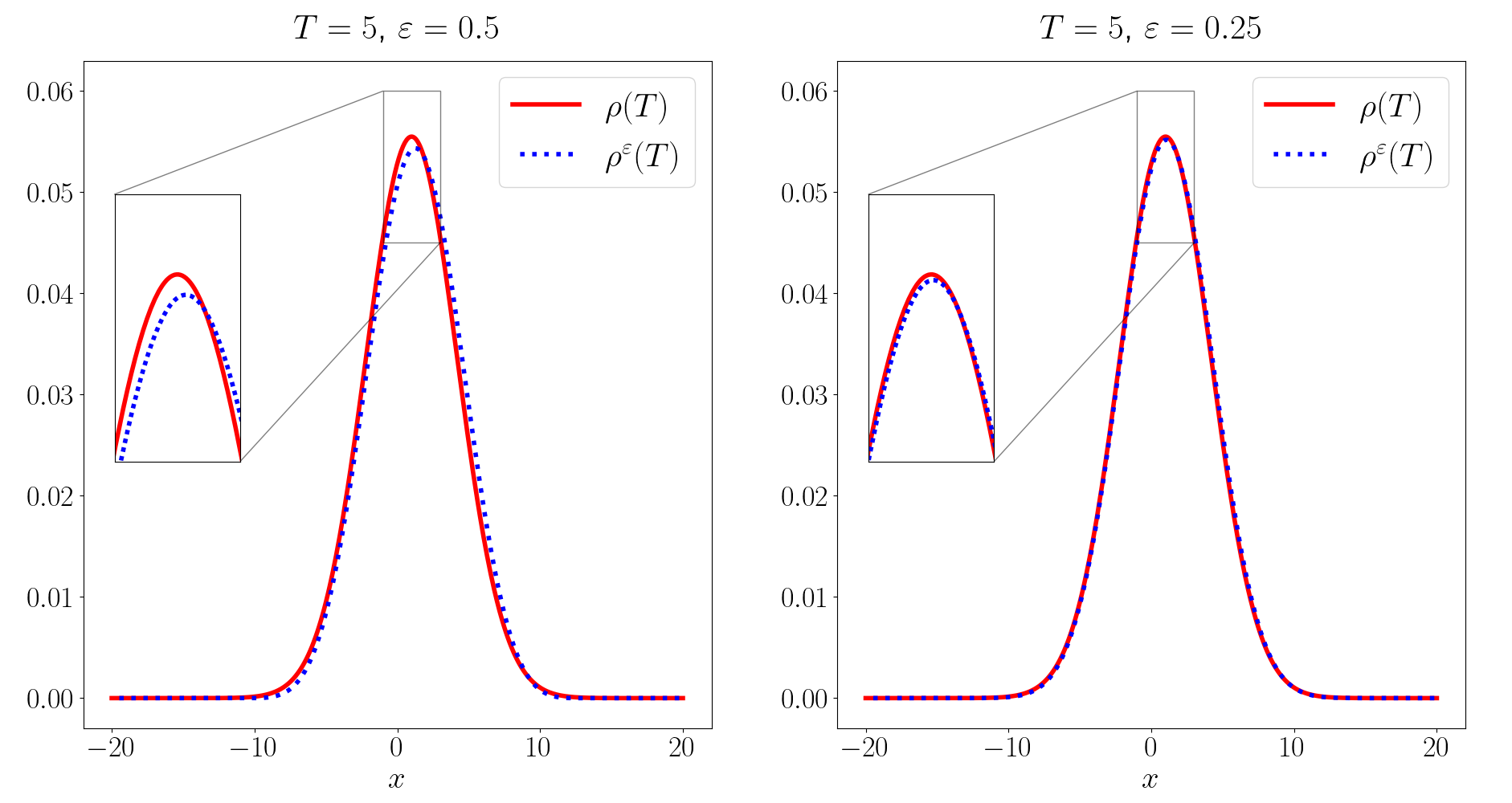}
    \caption{The first component $\rho^\eps$ of the solution of \eqref{mainr:Systn} (blue) approximates the solution $\rho$ of the heat equation \eqref{mainR:heat} (red) as $\eps\to 0$. The plots were generated for $h=2^{-4}$ and $T=5$.}
    \label{fig:overlapped}
\end{figure}

The plot in Figure \ref{fig:errorEps} serves as experimental evidence, indicating that for the initial data
\begin{equation}
\label{def:particularInitialData2}
\widetilde{\rho_0}(x)=e^{-\frac{1}{1-(x-1)^2}}\chi_{(0,2)}(x) \quad\text{ and }\quad\widetilde{u_0}(x)=e^{-\frac{1}{1-(x-1.5)^2}}\chi_{(0.5,2.5)}(x),
\end{equation}
the convergence rate of both the first and the third left-hand side term in \eqref{StrongRelax} is exactly $\mathcal{O}(\varepsilon^2)$, thus proving the sharpness of the rate in Theorem \ref{thm:NumRelax}. Moreover, the table in Figure \ref{fig:errorUniformH} confirms that the relaxation is uniform with respect to the grid width $h$.

\begin{figure}[h!]
    \centering
    \includegraphics[scale=0.22]{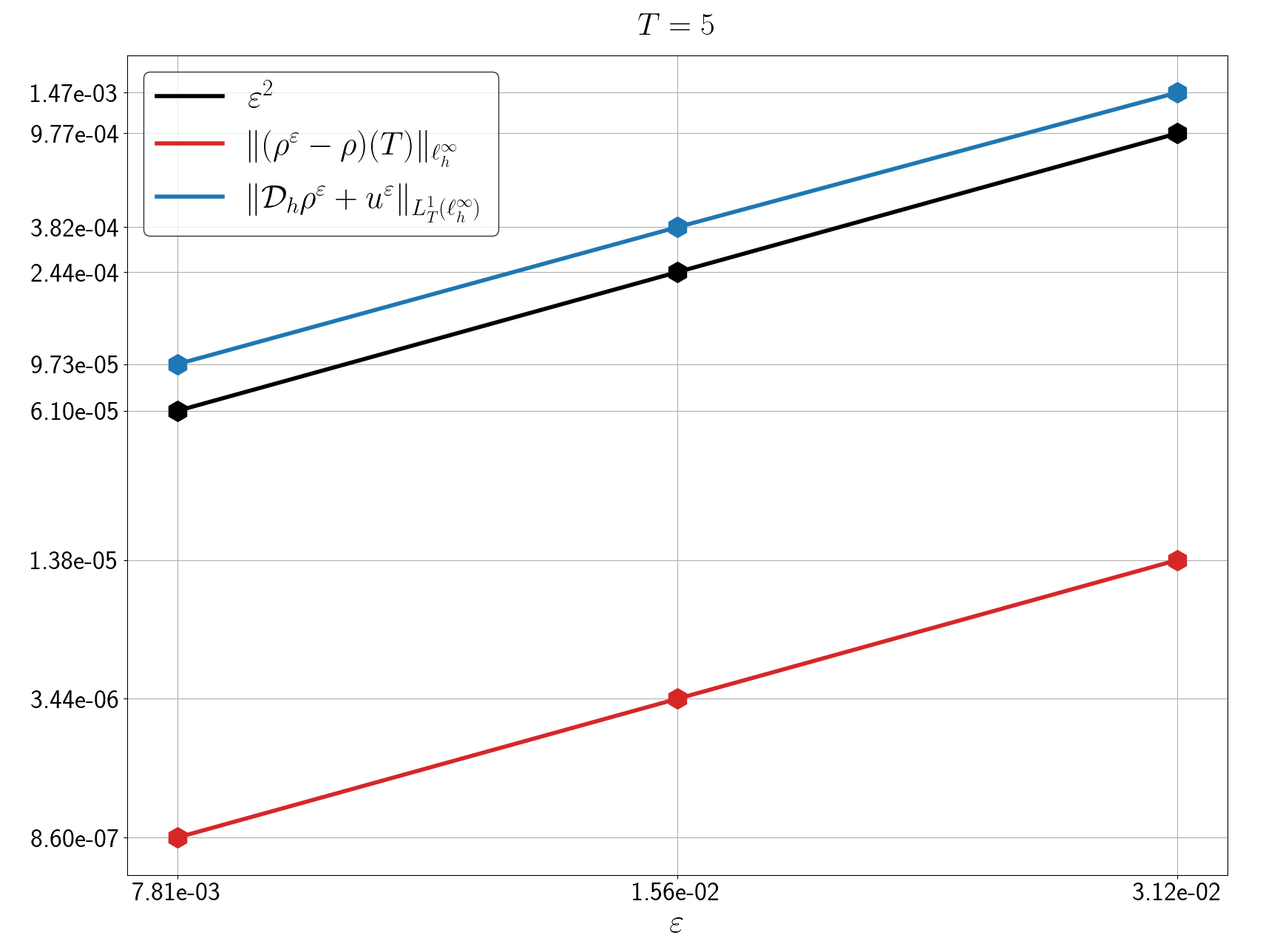}
    \caption{The log-log plot of the approximation error and Darcy law in $l^\infty_h$, obtained in Corollary \ref{cor:lInftyRelax}, as a function of $\varepsilon$, for fixed $h=2^{-4}$ and $T=5$.}
    \label{fig:errorEps}
\end{figure}
\begin{figure}[h!]
    \centering
    \begin{tabular}{|c|c|c|}
        \hline
        $h$ & $\|(\rho^\varepsilon-\rho)(T)\|_{\ell_h^\infty}$ & $\|\ID_h\rho^\varepsilon+u^\varepsilon\|_{L^1_T(\ell^\infty_{h})}$ \\[3pt]
         \hline
         $2^{-4}$ & $1.375812666{\rm e-}05$ & $1.468560202{\rm e-}03$ \\ 
         \hline
         $2^{-5}$ & $1.376071039{\rm e-}05$ & $1.525401187{\rm e-}03$\\ 
         \hline
         $2^{-6}$ & $1.376255148{\rm e-}05$ & $1.537860425{\rm e-}03$ \\ 
         \hline
    \end{tabular}
    \caption{The approximation error and the Darcy law in $l^\infty_h$, obtained in Corollary \ref{cor:lInftyRelax}, in terms of $h$, for fixed $\varepsilon=2^{-5}$ and $T=5$.}
    \label{fig:errorUniformH}
\end{figure}

\vfill 

\pagebreak

\section{Conclusion and extensions} \label{sec:ext}
Our theoretical and experimental evidence demonstrates that the decay estimates and relaxation properties inherent to partially dissipative hyperbolic systems can be effectively captured by one of the simplest and unconditionally stable numerical techniques: the central finite difference scheme. We have introduced a novel approach for numerically approximating the solutions of a class of parabolic equations (the continuous counterpart of \eqref{Z1withW}), utilizing only first-order discrete operators. The new discrete Littlewood-Paley theory we propose may serve as a foundation for addressing other problems related to other discrete equations, particularly those in which frequency decomposition techniques play a central role.
\medbreak

Dedicated to broadening the scope of Theorem \ref{thm:NumRelax}, we discuss some additional research directions.

\begin{enumerate}
\item[1.] \emph{Relaxation for general hyperbolic systems}. The convergence result derived in Theorem \ref{thm:NumRelax} can be extended to address the relaxation of general hyperbolic systems that satisfy the Kalman rank condition \eqref{K}. In fact, following the approach outlined in \cite{DanchinEMS}, under additional conditions on the matrix $A$ (see \cite[p.175]{DanchinEMS}), one can readily generalize our result to show the convergence, as $\varepsilon\to0$, of the discrete hyperbolic system 
   \begin{equation} \left\{ \begin{matrix}\displaystyle \partial_t U_1 + A_{1,1} \cD_h U_1 +A_{1,2}\cD_h U_2=0,\\ \varepsilon^2\left(\partial_t U_2 + A_{2,2}\cD_h U_2\right)+ A_{2,1}\cD_h U_1 =-\widetilde{B} U_2,\end{matrix} \right. \label{GE2}
\end{equation} 
towards the discrete diffusive system
\begin{equation}\label{Z1withW}\partial_tU_1 +  A_{1,1}\cD_hU_1-A_{1,2}\widetilde{B}^{-1}A_{2,1}\cD_h^2 U_1=0.
\end{equation}
Note that under the Kalman rank condition \eqref{K} or the Shizuta-Kawashima condition \cite[]{SK}, the operator $\mathcal{L}=A_{1,2}\widetilde{B}^{-1}A_{2,1}\d_{xx}^2$ was proven to be strongly elliptic in \cite{CBD3}.
    \smallbreak
    \item[2.] \emph{The Jin-Xin approximation}. For a conservation law:
\begin{equation}\label{syst:CL}
\begin{aligned}
  \d_t \rho +\d_xf(\rho)=0,
\end{aligned}
\end{equation}
its Jin-Xin approximation reads:
\begin{equation}\label{syst:JinXin}
\left\{\begin{aligned}
   & \d_t \rho^\varepsilon + \d_x u^\varepsilon=0,
    \\ & \varepsilon^2(\d_t u^\varepsilon + \d_x \rho^\varepsilon)=u^\varepsilon-f(\rho^\varepsilon).
\end{aligned}\right.
\end{equation}
This approximation was introduced in \cite{JinXin} and further examined through a frequency-decomposition approach in \cite{CBSJX}. It should be possible to justify the limit from the discrete approximation of \eqref{syst:JinXin}  to the discrete counterpart of \eqref{syst:CL} as $\varepsilon$ approaches zeros using the discrete frequency framework established here. The challenge further involves formulating product laws to handle the nonlinearity $f(\rho^\varepsilon)$ which, in the simplest case, reads as $(\rho^\varepsilon)^2$.

\end{enumerate}
\appendix
\section{Various lemmata} 
\subsection{Proof of Lemma \ref{lem:corrector}.} \label{A:lem:correct}
\begin{proof}
Differentiating $\cI(t)$ in time, we obtain
\begin{equation}\label{TimederivativeCorrector}
\begin{aligned}
  \frac d{dt}\cI(t) + \sum_{k=1}^{N-1} \varepsilon_k\|BA^k\ID_hU(t)\|_{l_h^2}^2 
 =& -\sum_{k=1}^{N-1} \varepsilon_k\bigl(BA^{k-1} BU,BA^k\ID_hU\bigr)_{l_h^2}
  \\&-\sum_{k=1}^{N-1} \varepsilon_k \bigl(BA^{k-1}U,BA^k B \ID_hU\bigr)_{l_h^2}
  \\&-\sum_{k=1}^{N-1} \varepsilon_k \bigl(BA^{k-1}U, BA^{k+1}\ID_h^2 U\bigr)_{l_h^2}\cdotp
\end{aligned}
\end{equation}
 To deal with the remainder terms, we proceed as in \cite{BZ,CBD2,Handbook} with some adaptations regarding the discrete setting. First, we fix a positive constant $\varepsilon_0$ and estimate the terms in the right-hand side of \eqref{TimederivativeCorrector} as follows. 
\begin{itemize}
\item The terms $\cI^1_k:= \varepsilon_k\bigl(BA^{k-1}B  U, BA^k\ID_h U\bigr)_{l_h^2}$
with $k\in\{1,\cdots, N-1\}$: due to $B U=\widetilde{B}  U_{2}$ and the fact that the matrices $A$, $B$ are bounded operators, we obtain
$$
 |\cI^1_k| \leq C\varepsilon_k\|\widetilde{B} U_{2}(t)\|_{l_h^2} \|BA^k \ID_h U(t)\|_{l_h^2} 
\leq \frac{\varepsilon_0}{4N}{\| U_{2}(t)\|_{l_h^2}^2}+\frac{C\varepsilon_k^2}{\varepsilon_0}\| BA^k   \ID_hU(t)\|_{l_h^2}^2.
$$
\item  The term $\cI^2_1:= \varepsilon_1\bigl(BU, BA B \ID_hU\bigr)_{l_h^2}$: one has
\begin{equation}\nonumber
|\cI^2_1| \leq C\varepsilon_1 \|\widetilde{B}U_{2}(t)\|_{l_h^2}\| \widetilde{B}\ID_hU_{2}(t)\|_{l_h^2}\\
\leq \frac{\varepsilon_0}{4N} \|U_{2}(t)\|_{l_h^2}^2+ \frac{C\varepsilon_1^2}{\var_{0}} \|\ID_hU_{2}(t)\|_{l_h^2}^2.
\end{equation}
\item The terms  $\cI^2_k:=\varepsilon_k\bigl(BA^{k-1}  U, BA^k B \ID_hU\bigr)_{l_h^2}$ with  $k\in\{2,\cdots, N-1\}$ if $N\geq3$: we deduce, after integrating by parts, that
$$\begin{aligned}
 |\cI^2_k|&=\varepsilon_k |\bigl(BA^{k-1}  \ID_h U, BA^k B U\bigr)_{\ell^2_h} |\leq C\varepsilon_k \|BA^{k-1} \ID_h U(t)\|_{l_h^2}\|B U(t)\|_{l_h^2}\\
&\leq \frac{\varepsilon_0}{4N}{\|  U_{2}(t)\|_{l_h^2}^2}+\frac{C\varepsilon_{k}^2}{\varepsilon_0}\| BA^{k-1}   \ID_hU(t)\|_{l_h^2}^2.
\end{aligned}$$
\item The terms $\cI^3_k:= \varepsilon_k \bigl( BA^{k-1}  U,  BA^{k+1}  \ID^2_hU\bigr)_{l_h^2}$ with $k\in\{1,\cdots, N-2\}$ if $N\geq3$: a similar argument yields
$$
|\cI^3_k|=\varepsilon_k | \bigl( BA^{k-1}  \ID_h U,  BA^{k+1}  \ID_h U\bigr)_{l_h^2}|
\leq \frac{\varepsilon_{k-1}}{8}  \|BA^{k-1}  \ID_hU(t)\|_{l_h^2}^2 +\frac{C\varepsilon_k^2}{\varepsilon_{k-1}}  \| BA^{k+1} \ID_h U(t)\|_{l_h^2}^2. 
$$
\item  The term $\cI^3_{N-1}:= \varepsilon_{N-1} \bigl( BA^{N-2}  U,  BA^{N}  \ID_h^2U\bigr)_{l_h^2}$:
owing to the Cayley-Hamilton theorem, there exist coefficients
$c_{*}^j$ ($j=0,1,2,...,N-1$) such that
\begin{align}
    &A^N =\sum_{j=0}^{N-1}  c_{*}^j A^j.\label{CH}
\end{align}
Consequently, one gets
$$\begin{aligned}
|\cI^3_{N-1}|&\leq \varepsilon_{N-1}\sum_{j=0}^{N-1}c_{*}^j \| BA^{N-2}  \ID_hU(t)\|_{l_h^2}
\|BA^j  \ID_hU(t)\|_{l_h^2} \\
&\leq  \frac{\varepsilon_{N-2}}{8}\| BA^{N-2}  \ID_hU(t)\|_{l_h^2}^2+\sum_{j=1}^{N-1} \frac{C\varepsilon_{N-1}^2}{\varepsilon_{N-2}}\| BA^j  \ID_hU(t)\|_{l_h^2}^2+\frac{C\varepsilon_{N-1}^2}{\varepsilon_{N-2}} \|\ID_h U_{2}(t)\|_{l_h^2}^2 .
\end{aligned} $$
\end{itemize}
In order to absorb the right-hand side terms of $\mathcal{I}_k^1$ and $\mathcal{I}_k^2$ by the left-hand side of \eqref{TimederivativeCorrector}, we take the constant  $\varepsilon_k$ small enough so that 
\begin{equation}\label{eq:e1}
C\varepsilon_1^2\leq \frac{\var_{0}^2}{8},\quad\quad C\varepsilon_k^2\leq \frac{\varepsilon_k\var_{0}}{8},\quad\quad k=1,2,...,N-1. \end{equation}
To handle the above estimates of  $\cI^3_k$ with $k=1,2,...,N-2$, one may let
\begin{equation}\label{eq:e11}
 C\varepsilon_k^2\leq \frac{1}{8}\varepsilon_{k-1}\varepsilon_{k+1},\quad\quad k=1,2,...,N-2\quad\text{if}~N\geq3.
 \end{equation}
In addition, to handle the term $\cI^3_{N-1}$, we assume 
\begin{equation}\label{eq:e2}
C\varepsilon_{N-1}^2\leq \frac{1}{8}\varepsilon_{j}\varepsilon_{N-2},\qquad j=0,\cdots,N-1.
\end{equation}
Clearly, the inequality \eqref{ddtI} holds if we find $\varepsilon_1,\cdots,\varepsilon_{N-1}$ fulfilling 
\eqref{eq:e1} -- \eqref{eq:e2}. 
As in \cite{BZ}, one can take  $\varepsilon_k=\varepsilon^{m_k}$ with some suitably small constant  $\varepsilon\leq \var_{0}$  and 
$m_1,\cdots,m_{N-1}$ satisfying for some $\delta>0$ (that can be taken arbitrarily small): 
$$m_{k}>1,\quad m_k\geq \frac{m_{k-1}+m_{k+1}}2+\delta \quad \text{and} \quad
m_{N-1}\geq\frac{m_{k}+m_{N-2}}2+ \delta, \quad k=1,\cdots,N-2.$$
This concludes the proof of Lemma \ref{lem:corrector}.
\end{proof}

\subsection{Technical lemmata}
\begin{lemma}[{\cite[Lemma A.1]{CBD2}}]\label{SimpliCarre}
Let $X : [0,T]\to \mathbb{R}_+$ be a continuous function such that $X^2$ is differentiable. Assume that there exist 
 a constant $\beta\geq 0$ and a measurable function $ \alpha: [0,T]\to \mathbb{R}_+$ 
such that 
 $$\frac{1}{2}\frac{d}{dt}X^2+\beta X^2\leq \alpha X\quad\hbox{a.e.  on }\ [0,T].$$ 
 Then, for all $t\in[0,T],$ we have
$$X(t)+\beta\int_0^tX\leq X_0+\int_0^t\alpha.$$
\end{lemma}
\begin{lemma}
    \label{A:lem:estimateExponential}
    Let $\lambda$ be a positive constant. There exists a constant $C=C(\lambda)>0$ such that, for every $t\in(0,\infty)$,
    \[I(t):=\int_{0}^{t}e^{-\lambda (t-\tau)}(1+\tau)^{-\frac{1}{2}}\,\dd\tau\leq C(1+t)^{-\frac{1}{2}}.\]
\end{lemma}
\begin{proof}
An integration-by-parts argument leads to
\[I(t)=\frac{1}{\lambda}(1+t)^{-\frac{1}{2}}-\frac{1}{\lambda}e^{-\lambda t}+\frac{1}{2\lambda}\int_0^t e^{-\lambda(t-\tau)} \frac{1}{(1+\tau)^\frac{3}{2}}\,\dd\tau.\]
The conclusion follows since the last term in the previous inequality can be bounded from above by
\[\int_0^{\frac{t}{2}} e^{-\lambda(t-\tau)}\,\dd\tau+ \int_{\frac{t}{2}}^t \frac{1}{(1+\tau)^\frac{3}{2}}\,\dd\tau.\]
\end{proof}

\noindent\textbf{Acknowledgments}
\smallbreak
\noindent The authors are grateful to Prof. Enrique Zuazua for his
valuable comments on a preliminary version of this paper. 
\smallbreak
\noindent\textbf{Funding}
\smallbreak
\noindent T. Crin-Barat has been funded by the Alexander von Humboldt-Professorship program and the Transregio 154 Project “Mathematical Modelling, Simulation and Optimization Using the Example of Gas Networks” of the DFG. 

D. Manea has been partially supported by the Romanian Ministry of Research, Innovation and Digitization, CNCS - UEFISCDI, project number PN-III-P1-1.1-TE-2021-1539. 

Part of this research was carried out while D. Manea was visiting the Chair of Dynamics, Control, Machine Learning and Numerics. The aforementioned research visit was funded by the Cost Action CA18232 MAT-DYN-NET, supported by COST (European Cooperation in Science and Technology).



\bibliographystyle{agsm}
\bibliography{ref}

\vfill 

\end{document}